\documentclass[a4paper,english,oneside]{article}
\usepackage[T1]{fontenc}
\usepackage[utf8]{inputenc} 

\usepackage{geometry}
\usepackage{amsmath}
\usepackage{amsfonts}
\usepackage{amssymb}
\usepackage{amsthm}
\usepackage[english]{babel}
\usepackage{color}
\usepackage{hyperref}
\hypersetup{hidelinks}

\usepackage[cal=boondoxo,bb=ams]{mathalfa}

\theoremstyle{plain}
\newtheorem{theorem}{Theorem}[section]
\newtheorem{lemma}[theorem]{Lemma}
\newtheorem{corollary}[theorem]{Corollary}
\newtheorem{proposition}[theorem]{Proposition}

\theoremstyle{definition}
\newtheorem{definition}[theorem]{Definition}

\theoremstyle{remark}
\newtheorem{remark}{Remark}

 \date{}
     
\begin{document}

\title{Blow -- up results for semilinear damped wave equations in Einstein~--~de~Sitter spacetime}

\author{Alessandro Palmieri}
\maketitle

\begin{abstract}
We prove by using an iteration argument some blow-up results for a semilinear damped wave equation in generalized Einstein-de Sitter spacetime with a time-dependent coefficient for the damping term and power nonlinearity. Then, we conjecture an expression for the critical exponent due to the main blow-up results, which is consistent with many special cases of the considered model and provides a natural generalization of Strauss exponent. In the critical case, we consider a non-autonomous and parameter dependent Cauchy problem for a linear ODE of second order, whose explicit solutions are determined by means of special functions' theory.
\end{abstract}

\begin{flushleft}
\end{flushleft}

\begin{flushleft}
\textbf{AMS Classification (2020)} Primary: 35B44, 35L05, 35L71; Secondary: 35B33, 33C10
\end{flushleft}


\section{Introduction}

In recent years, the wave equation in Einstein --  de Sitter spacetime has been considered in \cite{GalKinYag10,GalYag14} in the linear case and in \cite{GalYag15,GalYag17EdS,PTY20} in the semilinear case. Let us consider the semilinear wave equation with power nonlinearity in a \emph{generalized Einstein -- de Sitter spacetime}, that is, the equation with singular coefficients
\begin{align} \label{Semi EdeS k Psi}
\varphi_{tt} -t^{-2k} \Delta \varphi +2t^{-1} \varphi_t =|\varphi|^p,
\end{align} where $k\in [0,1)$ and $p>1$. This model is the semilinear wave equation in  Einstein -- de Sitter spacetime  with power nonlinearity for $k=2/3$ and $n=3$. It has been proved in \cite{GalYag17EdS,PTY20}  that for $$1<p\leqslant\max\big\{p_0\big(k,n+\tfrac{2}{1-k}\big),p_1(k,n)\big\}$$ a local in time solution to the corresponding Cauchy problem (with initial data prescribed at the initial time $t=1$) blows up in finite time, provided that the initial data fulfill certain integral sign conditions. More specifically, in \cite{GalYag17EdS} the subcritical case for \eqref{Semi EdeS k Psi} is investigated, while in \cite{PTY20} the critical case and the upper bound estimates for the lifespan are studied. Here and throughout the paper $p_0(k,n)$ is the positive root of the quadratic equation
\begin{align}\label{intro equation critical exponent general case}
\left(\tfrac{n-1}{2} -\tfrac{k}{2(1-k)}\right)p^2- \left(\tfrac{n+1}{2} +\tfrac{3k}{2(1-k)}\right)p -1=0,
\end{align} when the coefficient for $p^2$ is not positive, we set formally $p_0(k,n)\doteq \infty$, 
 while 
\begin{align} \label{intro def p1}
p_1(k,n) \doteq 1+ \frac{2}{(1-k)n}.
\end{align} Note that $p_1(k,n)$ is related to the Fujita exponent $p_{\mathrm{Fuj}}(n)\doteq 1+\frac{2}{n}$. Indeed, according to this notation, it holds $p_1(k,n)=p_{\mathrm{Fuj}}\big((1-k)n\big)$ and $p_1(0,n)=p_{\mathrm{Fuj}}(n)$. On the other hand, $p_0(k,n)$ is a generalization of the Strauss exponent for the classical semilinear wave equation, since $p_0(0,n)=p_{\mathrm{Str}}(n)$, where $p_{\mathrm{Str}}(n)$ is the positive root of the quadratic equation $(n-1)p^2-(n+1)p-2=0$.

In this paper, we generalize the model \eqref{Semi EdeS k Psi} with a general multiplicative constant $\mu$ for the damping term. More specifically, we investigate the blow -- up dynamic for the Cauchy problem 
\begin{align}\label{Semi EdeS k damped} 
\begin{cases}  u_{tt} - t^{-2k}\Delta u+\mu \, t^{-1} u_t= |u|^p & x\in \mathbb{R}^n, \ t\in (1,T), \\
u(1,x)= \varepsilon u_0(x) & x\in \mathbb{R}^n, \\
 u_t(1,x)= \varepsilon u_1(x) & x\in \mathbb{R}^n,
\end{cases}
\end{align} where $k\in [0,1)$, $p>1$, $\mu$ is the nonnegative multiplicative constant in the time -- dependent coefficient for the damping term and $\varepsilon>0$ describes the size of the initial data. Let us point out that the not damped case $\mu=0$ can be treated as well via our approach.

 More precisely,  we will focus on proving  blow-up results whenever the exponent $p$ belongs to the range  $$1<p\leqslant \max\big\{p_0\big(k,n+\tfrac{\mu}{1-k}\big),p_1(k,n)\big\},$$ clearly, under suitable sign assumptions for $u_0,u_1$. According to \eqref{intro equation critical exponent general case}, the shift $p_0\big(k,n+\tfrac{\mu}{1-k}\big)$ of $p_0(k,n)$ is nothing but the positive root to the quadratic equation
 \begin{align}\label{intro equation critical exponent shifted}
\left(\tfrac{n-1}{2} +\tfrac{\mu-k}{2(1-k)}\right)p^2- \left(\tfrac{n+1}{2} +\tfrac{\mu+3k}{2(1-k)}\right)p -1=0.
\end{align} Therefore, the critical exponent $p_0\big(k,n+\tfrac{\mu}{1-k}\big)$ for \eqref{Semi EdeS k damped} is obtained by the corresponding exponent in the not damped case via a formal shift in the dimension of magnitude $\tfrac{\mu}{1-k}$.

Let us provide  an overview on the methods that we are going to use to prove the main results in this paper. In the subcritical case $1<p< \max\big\{p_0\big(k,n+\tfrac{\mu}{1-k}\big),p_1(k,n)\big\}$, we employ a standard iteration argument based on a multiplier argument (see also \cite{LT18Scatt,LT18Glass,LT18ComNon,LinTu19} for further details on the multiplier argument). This approach is based on the employment of two time -- dependent functionals related to a local solution $u$ to \eqref{Semi EdeS k damped} and generalizes the method from \cite{TuLin17} for the semilinear wave equation with scale -- invariant damping. The first functional is the space average of $u$ and its dynamic will be considered for the iterative argument. On the other hand, we will work with a positive solution of the adjoint linear equation in order to prove the positivity of the second auxiliary functional. Hence, this second functional will also provide a first lower bound estimate for the first functional, allowing us to begin with the iteration procedure. In the critical case we should sharpen our iteration frame by considering a different time -- dependent functional, so that a slicing procedure may be applied. In comparison to what happens in the subcritical case, a more precise analysis of the adjoint linear equation is necessary in the critical case $p=p_0\big(k,n+\tfrac{\mu}{1-k}\big)$. This approach follows the one developed in \cite{PTY20} which is in turn a generalization of the ideas introduced by Wakasa and Yordanov in \cite{WakYor18,WakYor18Damp} an developed in different frameworks in \cite{PalTak19,PalTak19mix,LinTu19,ChenPal19MGT,ChenPal19SWENM}.
Whereas in the other critical case $p=p_1(k,n)$, we can still work with the space average of a local in time solution as functional, although a slicing procedure has to be applied in order to deal with logarithmic factors in the lower bound estimates.

\subsection{Notations}

Throughout this paper we use the following notations: $\phi_k(t)\doteq \frac{t^{1-k}}{1-k}$ denotes a distance function produced by the speed of propagation $a_k(t)=t^{-k}$, while the amplitude of the light cone is given by the function 
\begin{align} \label{def A k}
A_k(t)\doteq \int_1^t \tau^{-k} \mathrm{d}\tau = \phi_k(t) -\phi_k(1);
\end{align} the ball in $\mathbb{R}^n$ with radius $R$ around the origin is denoted $B_R$; $f\lesssim g$ means that there exists a positive constant $C$ such that $f\leqslant C g$ and, similarly, for $f\gtrsim g$; $\mathrm{I}_\nu$ and $\mathrm{K}_\nu$ denote the modified Bessel function of first and second kind of order $\nu$, respectively; finally, as in the introduction, $p_0(k,n)$ is the positive solution to \eqref{intro equation critical exponent general case} and $p_1(k,n)$ is defined by \eqref{intro def p1}.

\subsection{Main results}


Before stating the main theorems, let us introduce a suitable notion of energy solution to the semilinear Cauchy problem \eqref{Semi EdeS k damped}.

\begin{definition} \label{Def energy sol} Let $u_0\in H^1(\mathbb{R}^n)$ and $u_1\in L^2(\mathbb{R}^n)$. We say that 
\begin{align*}
u\in \mathcal{C} \big([1,T), H^1(\mathbb{R}^n)\big) \cap \mathcal{C}^1 \big([1,T), L^2(\mathbb{R}^n)\big)\cap L^p_{\mathrm{loc}}\big([1,T)\times \mathbb{R}^n\big)
\end{align*} is an energy solution to \eqref{Semi EdeS k damped} on $[1,T)$ if $u$ fulfills $u(1,\cdot)=\varepsilon u_0$ in $H^1(\mathbb{R}^n)$ and the integral relation
\begin{align}
& \int_{\mathbb{R}^n}  \partial_t u(t,x) \psi (t,x) \,  \mathrm{d}x  -\varepsilon \int_{\mathbb{R}^n} u_1(x)  \psi (1,x) \,  \mathrm{d}x -\int_1^t\int_{\mathbb{R}^n}     \partial_t u (s,x) \psi_{s}(s,x) \, \mathrm{d}x \, \mathrm{d}s \notag \\ & \qquad 
+ \int_1^t\int_{\mathbb{R}^n}  s^{-2k} \nabla u (s,x) \cdot\nabla \psi (s,x) \, \mathrm{d}x \, \mathrm{d}s+ \int_1^t\int_{\mathbb{R}^n} \mu s^{-1} \partial_t u (s,x)\psi(s,x) \, \mathrm{d}x \, \mathrm{d}s \notag  \\ 
&\quad  = \int_1^t\int_{\mathbb{R}^n}  |u(s,x)|^p \psi (s,x) \, \mathrm{d} x \, \mathrm{d} s \label{integral identity def energy sol}
\end{align} for any test function $\psi\in \mathcal{C}_0^\infty([1,T)\times \mathbb{R}^n)$ and any $t\in (1,T)$.
\end{definition}

We point out that performing a further step of integration by parts in \eqref{integral identity def energy sol}, we find the integral relation
\begin{align}
& \int_{\mathbb{R}^n}  \partial_t u(t,x) \psi (t,x) \,  \mathrm{d}x - \int_{\mathbb{R}^n} u(t,x)\psi_s(t,x) \, \mathrm{d}x  + \int_{\mathbb{R}^n} \mu\, t^{-1} u(t,x) \psi(t,x) \,\mathrm{d}x \notag \\
& \quad -\varepsilon \int_{\mathbb{R}^n} u_1(x)  \psi (1,x) \,  \mathrm{d}x + \varepsilon \int_{\mathbb{R}^n} u_0(x) \psi_s(1,x) \, \mathrm{d}x  - \varepsilon \int_{\mathbb{R}^n} \mu\,  u_0(x) \psi(1,x) \,\mathrm{d}x \notag \\ & \quad 
+\int_1^t\int_{\mathbb{R}^n}   u (s,x)  \left( \psi_{ss}(s,x) -s^{-2k} \Delta \psi (s,x)-\mu s^{-1} \psi_s(s,x)+\mu s^{-2}\psi(s,x)\right) \mathrm{d}x \, \mathrm{d}s \notag  \\ 
&\quad  = \int_1^t\int_{\mathbb{R}^n}  |u(s,x)|^p \psi (s,x) \, \mathrm{d} x \, \mathrm{d} s \label{integral identity  weak sol}
\end{align} for any $\psi\in \mathcal{C}_0^\infty([1,T)\times \mathbb{R}^n)$ and any $t\in (1,T)$.

\begin{remark}\label{Remark support} Let us point out that if the Cauchy data have compact support, say $\mathrm{supp}\, u_j \subset B_R$ for $j=0,1$ and for some $R>0$, then, for any $t\in (1,T)$ a local solution $u$ to \eqref{Semi EdeS k damped}  the support condition $$\mathrm{supp} \, u(t, \cdot) \subset B_{R+A_k(t)} $$ is satisfied, where $A_k$ is defined by \eqref{def A k}. Consequently, in Definition \ref{Def energy sol} it is possible to consider test functions which are not compactly supported, i.e., $\psi \in \mathcal{C}^\infty([1,T)\times \mathbb{R}^n)$. 
\end{remark}

\begin{theorem}[Subcritical case]\label{Theorem subcritical case}
Let $\mu\geqslant 0$ and let the exponent of the nonlinear term $p$ satisfy $$1<p< \max\left\{p_0\big(k,n+\tfrac{\mu}{1-k}\big),p_1(k,n)\right\}.$$ Let us assume that $u_0\in H^1(\mathbb{R}^n)$ and $u_1\in L^2(\mathbb{R}^n)$ are nonnegative, nontrivial and compactly supported functions with supports contained in $B_R$ for some $R>0$. Let $$u\in \mathcal{C} \big([1,T), H^1(\mathbb{R}^n)\big) \cap \mathcal{C}^1 \big([1,T), L^2(\mathbb{R}^n)\big)\cap L^p_{\mathrm{loc}}\big([1,T)\times \mathbb{R}^n\big)$$ be an energy solution to \eqref{Semi EdeS k damped} according to Definition \ref{Def energy sol}  with lifespan $T=T(\varepsilon)$ and satisfying the support condition $\mathrm{supp} \, u(t,\cdot)\subset B_{A_k(t)+R}$ for any $t\in (1,T)$.

 Then, there exists a positive constant $\varepsilon_0~=~\varepsilon_0(u_0,u_1,n,p,k,\mu,R)$ such that for any $\varepsilon\in (0,\varepsilon_0]$ the energy solution $u$ blows up in finite time. Moreover, the upper bound estimate for the lifespan
\begin{align}\label{upper bound est lifespan subcrit thm}
T(\varepsilon)\leqslant \begin{cases} C\varepsilon^{-\frac{p(p-1)}{\theta(n,k,\mu,p)}} & \mbox{if} \ p<p_0\big(k,n+\tfrac{\mu}{1-k}\big), \\ C\varepsilon^{-\left(\frac{2}{p-1}-(1-k)n\right)^{-1}} & \mbox{if} \ p<p_1(k,n),
\end{cases}
\end{align} holds, where the positive constant $C$ is independent of $\varepsilon$ and
\begin{align*}
\theta(n,k,\mu,p)\doteq 1-k+\left((1-k)\tfrac{n+1}{2}+\tfrac{\mu+3k}{2} \right)p-\left((1-k)\tfrac{n-1}{2}+\tfrac{\mu-k}{2} \right)p^2.
\end{align*}
\end{theorem}

In order to properly state the results in the critical case, let us explicit provide the threshold for $\mu$ which yields the transition from a dominant $p_0\big(k,n+\tfrac{\mu}{1-k}\big)$ to the case in which $p_1(k,n)$ is the highest exponent. Due to the fact that $p_0\big(k,n+\tfrac{\mu}{1-k}\big)$ is the biggest solution of \eqref{intro equation critical exponent shifted}, we have that $p_1(k,n)>p_0\big(k,n+\tfrac{\mu}{1-k}\big)$ if and only if 
\begin{align*}
\left(\tfrac{n-1}{2} +\tfrac{\mu-k}{2(1-k)}\right)p_1(k,n)^2- \left(\tfrac{n+1}{2} +\tfrac{\mu+3k}{2(1-k)}\right)p_1(k,n) -1>0.
\end{align*} By straightforward computations, it follows that $p_1(k,n)>p_0\big(k,n+\tfrac{\mu}{1-k}\big)$ for $\mu >\mu_0(k,n)$, 
where
\begin{align} \label{def mu_0(k,n)}
\mu_0(k,n) \doteq \frac{(1-k)^2n^2+(1-k)(1+2k)n+2}{n(1-k)+2}.
\end{align} Note that for $k=0$ the splitting value $\mu_0(k,n)$ does coincide with the one for the semilinear wave equation with scale -- invariant damping in the flat case from the work \cite{IS17}.

\begin{theorem}[Critical case: part I] \label{Theorem critical case p0}
Let $0\leqslant\mu\leqslant \mu_0(k,n)$ such that $ \mu\leqslant k$ or $\mu\geqslant 2-k$. We consider $p=p_0\big(k,n+\tfrac{\mu}{1-k}\big)$. Let us assume that $u_0\in H^1(\mathbb{R}^n)$ and $u_1\in L^2(\mathbb{R}^n)$ are nonnegative, nontrivial and compactly supported functions with supports contained in $B_R$ for some $R>0$. Let $$u\in \mathcal{C} \big([1,T), H^1(\mathbb{R}^n)\big) \cap \mathcal{C}^1 \big([1,T), L^2(\mathbb{R}^n)\big)\cap L^p_{\mathrm{loc}}\big([1,T)\times \mathbb{R}^n\big)$$ be an energy solution to \eqref{Semi EdeS k damped} according to Definition \ref{Def energy sol}  with lifespan $T=T(\varepsilon)$ and satisfying the support condition $\mathrm{supp} \, u(t,\cdot)\subset B_{A_k(t)+R}$ for any $t\in (1,T)$.

 Then, there exists a positive constant $\varepsilon_0~=~\varepsilon_0(u_0,u_1,n,p,k,\mu,R)$ such that for any $\varepsilon\in (0,\varepsilon_0]$ the energy solution $u$ blows up in finite time. Moreover, the upper bound estimate for the lifespan
\begin{align*}
T(\varepsilon)\leqslant \exp\left(C\varepsilon^{-p(p-1)}\right)
\end{align*} holds, where the positive constant $C$ is independent of $\varepsilon$.
\end{theorem}
 
\begin{remark} It seems that the assumption in Theorem \ref{Theorem critical case p0} for the multiplicative constant $ \mu\leqslant k$ or $\mu\geqslant 2-k$ is technical, since it is due to the method we are going to apply for the proof.
\end{remark}

\begin{theorem}[Critical case: part II]   \label{Theorem critical case p1}
Let $\mu\geqslant \mu_0(k,n)$ and $p=p_1(k,n)$. Let us assume that $u_0\in H^1(\mathbb{R}^n)$ and $u_1\in L^2(\mathbb{R}^n)$ are nonnegative, nontrivial and compactly supported functions with supports contained in $B_R$ for some $R>0$. Let $$u\in \mathcal{C} \big([1,T), H^1(\mathbb{R}^n)\big) \cap \mathcal{C}^1 \big([1,T), L^2(\mathbb{R}^n)\big)\cap L^p_{\mathrm{loc}}\big([1,T)\times \mathbb{R}^n\big)$$ be an energy solution to \eqref{Semi EdeS k damped} according to Definition \ref{Def energy sol}  with lifespan $T=T(\varepsilon)$ and satisfying the support condition $\mathrm{supp} \, u(t,\cdot)\subset B_{A_k(t)+R}$ for any $t\in (1,T)$.

 Then, there exists a positive constant $\varepsilon_0~=~\varepsilon_0(u_0,u_1,n,p,k,\mu,R)$ such that for any $\varepsilon\in (0,\varepsilon_0]$ the energy solution $u$ blows up in finite time. Moreover, the upper bound estimate for the lifespan
\begin{align*}
T(\varepsilon)\leqslant \exp\left(C\varepsilon^{-(p-1)}\right)
\end{align*} holds, where the positive constant $C$ is independent of $\varepsilon$.
\end{theorem}

The remaining part of the paper is organized as follows: the proof of the result in the subcritical case (cf. Theorem \ref{Theorem subcritical case}) is carried out in Section \ref{Section subcritical case}; in Section \ref{Section critical case p0} we prove Theorem \ref{Theorem critical case p0} by generalizing the approach introduced in \cite{WakYor18}; finally, we show the proof of Theorem \ref{Theorem critical case p1} in Section \ref{Section critical case p1} via a standard slicing procedure.

\section{Subcritical case} \label{Section subcritical case}

In this section we are going to prove Theorem \ref{Theorem subcritical case}. Let $u$ be a local in time solution to \eqref{Semi EdeS k damped} and let us assume that the assumptions from the statement of  Theorem \ref{Theorem subcritical case} on $p$ and on the data are fulfilled. We will follow the multiplier approach introduced by \cite{LTW17} and then improved by \cite{TuLin17}, to derive a suitable iteration frame for the time -- dependent functional 
\begin{align}\label{def U}
U_0(t) \doteq \int_{\mathbb{R}^n} u(t,x) \, \mathrm{d}x.
\end{align} In order to obtain a first lower bound estimate for $U_0$ we will introduce a second time -- dependent functional,  following the main ideas of the pioneering paper \cite{YZ06} and adapting them to the case with time -- depend coefficients as in \cite{HeWittYin17,GalYag17EdS,TuLin17,PT18}.

The section is organized as follows: in Section \ref{Subsection sol adj eq subcrit} we determine a suitable positive solution to the adjoint homogeneous linear equation with separate variables, then, we use this function to derive a lower bound estimate for $U_0$ in Section \ref{Subsection 1st lb estimate subcrit}; in Sections \ref{Subsection iter frame subcrit} and \ref{Subsection iter ar subcrit} the derivation of the iteration frame and its application in an iterative argument are dealt with, respectively.

\subsection{Solution of the adjoint homogeneous linear equation} \label{Subsection sol adj eq subcrit}

In this section, we shall determine a particular positive solution  to the adjoint homogeneous linear equation
\begin{align}\label{adj eq hom lin subcrit}
\Psi_{ss}-s^{-2k} \Delta \Psi-\mu \, s^{-1}\Psi_s+\mu \, s^{-2} \Psi=0.
\end{align}
First of all, we recall the remarkable function
\begin{align}\label{def Yordanov-Zhang function}
\varphi (x) \doteq \begin{cases} \int_{\mathbb{S}^{n-1}} \mathrm{e}^{x\cdot \omega} \mathrm{d} \sigma_\omega & \mbox{if} \ n\geqslant 2, \\ \cosh x 
& \mbox{if} \ n=1, \end{cases} 
\end{align} introduced in \cite{YZ06} for the study of the critical semilinear wave equation. The main properties of this function that will used throughout this paper are the following: $\varphi$ is a positive and smooth function that satisfies $\Delta \varphi =\varphi$ and asymptotically behaves like $c_n |x|^{-\frac{n-1}{2}}\mathrm{e}^{|x|}$ as $|x|\to \infty$.

If we look for a solution to \eqref{adj eq hom lin subcrit} with separate variables, that is, we consider the ansatz $\Psi(s,x)=\varrho(s) \varphi(x)$, then, it suffices to find a positive solution to the ODE
\begin{align}\label{ODE rho}
\varrho''-s^{-2k}\varrho-\mu s^{-1} \varrho' +\mu s^{-2} \varrho=0.
\end{align} We perform the change of variable $\tau = \phi_k(s)$. By using
\begin{align*}
\varrho' & =t^{-k} \frac{\mathrm{d}\varrho}{\mathrm{d}\tau}, \qquad  \varrho'' =t^{-2k} \frac{\mathrm{d}^2\varrho}{\mathrm{d}\tau^2} -kt^{-1-k} \frac{\mathrm{d}\varrho}{\mathrm{d}\tau}, 
\end{align*} it follows with straightforward computations that $\varrho$ solves \eqref{ODE rho} if and only if 
\begin{align}\label{ODE rho wrt tau}
\frac{\mathrm{d}^2\varrho}{\mathrm{d}\tau^2}-\frac{k+\mu}{1-k} \, \frac{1}{\tau} \, \frac{\mathrm{d}\varrho}{\mathrm{d}\tau}+\left(\frac{\mu}{(1-k)^2}\, \frac{1}{\tau^2} -1\right) \varrho=0.
\end{align} To further simplify the previous equation, we carry out the transformation $\varrho(\tau)=\tau^\sigma \zeta(\tau)$, where $\sigma\doteq \frac{1+\mu}{2(1-k)}$. Hence, using 
\begin{align*}
\frac{\mathrm{d}\varrho}{\mathrm{d}\tau}(\tau) & =\sigma \tau^{\sigma-1} \zeta(\tau) + \tau^\sigma\frac{\mathrm{d}\zeta}{\mathrm{d}\tau} (\tau), \qquad \frac{\mathrm{d}^2\varrho}{\mathrm{d}\tau^2} = \sigma (\sigma-1) \tau^{\sigma-2} \zeta(\tau)+2\sigma \tau^{\sigma-1} \frac{\mathrm{d}\zeta}{\mathrm{d}\tau} (\tau) + \tau^\sigma\frac{\mathrm{d}^2\zeta}{\mathrm{d}\tau^2} (\tau),
\end{align*} we get that $\varrho$ is a solution to \eqref{ODE rho wrt tau} if and only if $\zeta$ solves
\begin{align}\label{ODE zeta wrt tau}
\tau^2\frac{\mathrm{d}^2\zeta}{\mathrm{d}\tau^2}-\left(2\sigma-\frac{k+\mu}{1-k} \right) \tau \frac{\mathrm{d}\zeta}{\mathrm{d}\tau}+\left[ \sigma\left(\sigma-1-\frac{k+\mu}{1-k}\right)+\frac{\mu}{(1-k)^2}-\tau^2\right] \zeta=0.
\end{align} Due to the choice of the parameter $\sigma$, equation \eqref{ODE zeta wrt tau} is nothing but a modified Bessel equation of order $\gamma\doteq \frac{\mu-1}{2(1-k)}$, that is,  \eqref{ODE zeta wrt tau} can be rewritten as
\begin{align*}
\tau^2\frac{\mathrm{d}^2\zeta}{\mathrm{d}\tau^2}- \tau \frac{\mathrm{d}\zeta}{\mathrm{d}\tau}
-(\gamma^2+\tau^2)\zeta=0.
\end{align*} If we pick the modified Bessel function of the second kind $\mathrm{K}_\gamma$ as solution to the previous equation, then, up to a negligible multiplicative constant, we found
\begin{align} \label{def rho}
\rho(s)\doteq s^{\frac{1+\mu}{2}}\mathrm{K}_\gamma\big(\phi_k(s)\big)
\end{align} as a positive solution to \eqref{ODE rho} and, in turn, \begin{align} \label{def Psi}
\Psi(s,x)\doteq \rho(s) \varphi(x) = s^{\frac{1+\mu}{2}}\mathrm{K}_\gamma\big(\phi_k(s)\big) \varphi(x)
\end{align} as a positive solution of the adjoint equation \eqref{adj eq hom lin subcrit}.

In the next sections, we will need to employ the asymptotic behavior of the function $\varrho=\varrho(t)$ for $t\to \infty$. Since $\mathrm{K}_\gamma(z) = \sqrt{\pi/(2z)} \left(\mathrm{e}^{-z}+O(z^{-1})\right)$ as $z\to \infty$ (cf. \cite{NIST10}), then, the following asymptotic estimate holds
\begin{align} \label{asymptotic rho}
\varrho(t) =  \sqrt{\frac{\pi}{2}} \  t^{\frac{k+\mu}{2}}\, \mathrm{e}^{-\phi_k(t)} \left(1+O(t^{-1+k})\right) \qquad \mbox{for} \ t\to \infty.
\end{align}

The solution $\Psi$ of the adjoint equation \eqref{adj eq hom lin subcrit} that we determined in this section will be employed in Section \ref{Subsection 1st lb estimate subcrit} to introduce a second time -- dependent functional with the purpose to establish a first lower bound estimate for $U_0$.

\subsection{Derivation of the iteration frame} \label{Subsection iter frame subcrit}

In this section we are going to determine the iteration frame for the functional $U_0=U_0(t)$ defined in \eqref{def U}. Let us choose as test function $\psi=\psi(s,x)$ in the integral relation \eqref{integral identity def energy sol} such that $\psi=1$ on the forward cone $\{(s,x)\in [1,t]\times \mathbb{R}^n: |x|\leqslant R+ A_k(s)\}$. Then,
\begin{align*}
\int_{\mathbb{R}^n}  \partial_t u(t,x) \,  \mathrm{d}x  -\varepsilon \int_{\mathbb{R}^n} u_1(x)   \,  \mathrm{d}x + \int_1^t\int_{\mathbb{R}^n} \mu s^{-1} \partial_t u(s,x)\, \mathrm{d}x \, \mathrm{d}s   = \int_1^t\int_{\mathbb{R}^n}  |u(s,x)|^p \, \mathrm{d} x \, \mathrm{d} s 
\end{align*} which can be rewritten as
\begin{align*}
U'_0(t)-U'_0(1) + \int_1^t \mu s^{-1}  U'_0(s) \, \mathrm{d}s   = \int_1^t\int_{\mathbb{R}^n}  |u(s,x)|^p \, \mathrm{d} x \, \mathrm{d} s.
\end{align*} Differentiating the last identity with respect to $t$, we get
\begin{align*}
U''_0(t)+\mu t^{-1} U'_0(t) = \int_{\mathbb{R}^n}  |u(t,x)|^p \, \mathrm{d} x.
\end{align*} Multiplying the previous equation by $t^{\mu}$, it follows
\begin{align*}
t^\mu U''_0(t)+\mu t^{\mu-1} U'_0(t) = \frac{\mathrm{d}}{\mathrm{d}t}\big( t^\mu U'_0(t)\big) =t^\mu \int_{\mathbb{R}^n}  |u(t,x)|^p \, \mathrm{d} x.
\end{align*} Integrating twice this relation over $[1,t]$, we find
\begin{align}\label{iter fram subcrit for 1st lb data}
U_0(t)& = U_0(1) +U'_0(1) \int_1^t \tau^{-\mu}\,  \mathrm{d}\tau + \int_1^t \tau^{-\mu}\int_1^\tau s^\mu  \int_{\mathbb{R}^n}  |u(s,x)|^p \, \mathrm{d} x\,  \mathrm{d}s \,  \mathrm{d}\tau. 
\end{align} On the one hand , from \eqref{iter fram subcrit for 1st lb data} we derive the lower bound estimate
\begin{align} \label{lb estimate U0 trivial}
U_0(t) \gtrsim \varepsilon,
\end{align}  where the unexpressed positive multiplicative constant depends on $u_0,u_1$ due to the nonnegativeness of $u_0,u_1$ and $U^{(j)}(1)=\varepsilon \int_{\mathbb{R}^n} u_j(x) \mathrm{d}x$ for $j\in\{0,1\}$. On the other hand, we obtain the estimate
\begin{align}\label{iter fram subcrit for 1st lb}
U_0(t)& \geqslant \int_1^t \tau^{-\mu}\int_1^\tau s^\mu  \int_{\mathbb{R}^n}  |u(s,x)|^p \, \mathrm{d} x\,  \mathrm{d}s \,  \mathrm{d}\tau \\
& \gtrsim \int_1^t \tau^{-\mu}\int_1^\tau s^\mu  (R+A_k(s))^{-n(p-1)}  (U_0(s))^p \,  \mathrm{d}s \,  \mathrm{d}\tau , \notag
\end{align} where in the second step we applied Jensen's inequality and the support property for $u(s,\cdot)$. Therefore, we proved the following iteration frame for $U_0$
\begin{align}
U_0(t)& \geqslant  C \int_1^t \tau^{-\mu}\int_1^\tau s^{\mu -(1-k)n(p-1)}  (U_0(s))^p\,  \mathrm{d}s \,  \mathrm{d}\tau  \label{iter fram subcrit U0}
\end{align} for a suitable positive constant $C=C(n,p,k)$ and for $t\geqslant 1$. In Section \ref{Subsection iter frame subcrit} we will employ \eqref{iter fram subcrit U0} to derive iteratively a sequence of lower bound estimates for $U_0$. However, we shall first derive in Section \ref{Subsection 1st lb estimate subcrit} another lower bound estimate for $U_0$ that will provide, together with \eqref{lb estimate U0 trivial}, the starting point for the iteration procedure. 

\subsection{First lower bound estimate for the functional} \label{Subsection 1st lb estimate subcrit}

Let $\Psi=\Psi(t,x)$ be the function defined by \eqref{def Psi}. Since this function is smooth and positive, by applying the integral relation \eqref{integral identity  weak sol} to $\Psi$ and using the fact that $\Psi$ solves the adjoint equation \eqref{adj eq hom lin subcrit}, we get
\begin{align*}
 0 & \leqslant  \int_1^t\int_{\mathbb{R}^n}  |u(s,x)|^p \Psi (s,x) \, \mathrm{d} x \, \mathrm{d} s \\ & = \int_{\mathbb{R}^n}  \partial_t u(t,x) \Psi (t,x) \,  \mathrm{d}x - \int_{\mathbb{R}^n} u(t,x)\Psi_s(t,x) \, \mathrm{d}x  + \int_{\mathbb{R}^n} \mu\, t^{-1} u(t,x) \Psi(t,x) \,\mathrm{d}x  \\
& \quad -\varepsilon \int_{\mathbb{R}^n} \left( \varrho(1) u_1(x) +(\mu\varrho(1)-\varrho'(1))u_0(x)\right) \varphi (x) \,  \mathrm{d}x .
\end{align*} If we introduce the auxiliary functional
\begin{align}\label{def U1}
U_1(t)\doteq \int_{\mathbb{R}^n} u(t,x) \Psi(t,x) \, \mathrm{d} x,
\end{align} then, from the last estimate we have
\begin{align}\label{1st ineq U1}
U_1'(t)-\frac{2\varrho'(t)}{\varrho(t)}\, U_1(t)+\mu \, t^{-1} U_1(t) \geqslant \varepsilon \int_{\mathbb{R}^n} \big( \varrho(1) u_1(x) +(\mu\varrho(1)-\varrho'(1))u_0(x)\big) \varphi (x) \,  \mathrm{d}x,
\end{align} where we applied the relation
\begin{align*}
U_1'(t)= \int_{\mathbb{R}^n}  \partial_t u(t,x) \Psi (t,x) \,  \mathrm{d}x +\int_{\mathbb{R}^n} u(t,x)\psi_s(t,x) \, \mathrm{d}x = \int_{\mathbb{R}^n}  \partial_t u(t,x) \Psi (t,x) \,  \mathrm{d}x +\frac{\varrho'(t)}{\varrho(t)}U_1(t).
\end{align*}
Let compute more explicitly the term on the right -- hand side of \eqref{1st ineq U1} and show its positiveness. By using the recursive identity $$\mathrm{K}'_\gamma(z)=-\mathrm{K}_{\gamma+1}(z)+\frac{\gamma}{z} \,\mathrm{K}_\gamma(z)$$ for the derivative of the modified Bessel function of the second kind and $\gamma=\frac{\mu-1}{2(1-k)}$, it follows 
\begin{align*}
\varrho'(t) & = \tfrac{1+\mu}{2} \,t^{\frac{\mu-1}{2}} \mathrm{K}_\gamma\big(\phi_k(t)\big)+t^{\frac{1+\mu}{2}-k} \mathrm{K}'_\gamma\big(\phi_k(t)\big) \\ &=  \tfrac{1+\mu}{2}\, t^{\frac{\mu-1}{2}} \mathrm{K}_\gamma\big(\phi_k(t)\big)+t^{\frac{1+\mu}{2}-k} \Big(-\mathrm{K}_{\gamma+1}\big(\phi_k(t)\big) +\tfrac{\mu-1}{2} \, t^{-1+k}\mathrm{K}_\gamma\big(\phi_k(t)\big) \Big) \\
 &=  \mu  \, t^{\frac{\mu-1}{2}} \mathrm{K}_\gamma\big(\phi_k(t)\big)- t^{\frac{1+\mu}{2}-k} \mathrm{K}_{\gamma+1}\big(\phi_k(t)\big) .
\end{align*}  In particular, the following relations hold
\begin{align*}
\mu\varrho(1)-\varrho'(1) = \mathrm{K}_{\gamma+1}\big(\phi_k(1)\big)>0, \qquad \varrho(1)=\mathrm{K}_{\gamma}\big(\phi_k(1)\big)>0,
\end{align*} so that we may rewrite \eqref{1st ineq U1} as
\begin{align}
U_1'(t)-\frac{2\varrho'(t)}{\varrho(t)}\, U_1(t)+\mu \, t^{-1} U_1(t) & \geqslant \varepsilon \underbrace{\int_{\mathbb{R}^n} \big( \mathrm{K}_{\gamma}\big(\phi_k(1)\big)  u_1(x) +\mathrm{K}_{\gamma+1}\big(\phi_k(1)\big) u_0(x)\big) \varphi (x) \,  \mathrm{d}x }_{ \doteq I_{k,\mu}[u_0,u_1]} . \label{2nd ineq U1}
\end{align} Multiplying \eqref{2nd ineq U1} by $t^\mu/\varrho^2(t)$, we have
\begin{align*}
\frac{\mathrm{d}}{\mathrm{d}t} \left( \frac{t^\mu}{\varrho^2(t)} \,U_1(t)\right) = \frac{t^\mu}{\varrho^2(t)} U_1'(t)-\frac{2\varrho'(t)}{\varrho^3(t)}\, t^\mu U_1(t)+\mu \, t^{\mu-1} \frac{1}{\varrho^2(t)}U_1(t) \geqslant \varepsilon I_{k,\mu}[u_0,u_1]  \frac{t^\mu}{\varrho^2(t)}.
\end{align*} Integrating the previous inequality over $[1,t]$ and using the sign assumption on $u_0$, we get 
\begin{align*}
U_1(t)&\geqslant \frac{\varrho^2(t) t^{-\mu}}{\varrho^2(1)} U_1(1)+ \varepsilon I_{k,\mu}[u_0,u_1] \, \frac{\varrho^2(t)}{t^\mu}  \int_1^t  \frac{s^\mu}{\varrho^2(s)} \, \mathrm{d}s \\ &\geqslant \varepsilon I_{k,\mu}[u_0,u_1] \, \frac{\varrho^2(t)}{t^\mu} \int_1^t  \frac{s^\mu}{\varrho^2(s)} \, \mathrm{d}s.
\end{align*} Thanks to \eqref{asymptotic rho}, there exists $T_0=T_0(k,\mu)>1$ such that
\begin{align*}
U_1(t) &\gtrsim \varepsilon I_{k,\mu}[u_0,u_1] \, t^{k} \mathrm{e}^{-2\phi_k(t)} \int_{T_0}^t s^{-k} \mathrm{e}^{2\phi_k(s)} \, \mathrm{d}s
\end{align*}  for $t\geqslant T_0$. Consequently, for for $t\geqslant 2T_0$, shrinking the domain of integration in the last inequality, we have
\begin{align}
U_1(t) &\gtrsim \varepsilon I_{k,\mu}[u_0,u_1] \, t^{k} \mathrm{e}^{-2\phi_k(t)} \int_{t/2}^t s^{-k} \mathrm{e}^{2\phi_k(s)} \, \mathrm{d}s   
=2^{-1} \varepsilon I_{k,\mu}[u_0,u_1] \, t^{k} \mathrm{e}^{-2\phi_k(t)} \left( \mathrm{e}^{2\phi_k(t)}-\mathrm{e}^{2\phi_k(\frac t2)} \right)  \notag  \\
& = 2^{-1} \varepsilon I_{k,\mu}[u_0,u_1] \, t^{k}  \left( 1-\mathrm{e}^{2\phi_k(\frac t2)-2\phi_k(t)} \right) = 2^{-1} \varepsilon I_{k,\mu}[u_0,u_1] \, t^{k}  \left( 1-\mathrm{e}^{-\frac{2}{1-k}(1-2^{k-1})t^{1-k}} \right) \notag \\ & \geqslant 2^{-1} \varepsilon I_{k,\mu}[u_0,u_1] \, t^{k}  \left( 1-\mathrm{e}^{-\frac{2}{1-k}(2^{1-k}-1)T_0^{1-k}} \right) \gtrsim \varepsilon t^k. \label{lb estimate U1}
\end{align}

By repeating exactly the same computations as in \cite[Section 3]{PalRei18} (which are completely independent of the amplitude function $A_k$), we obtain 
\begin{align*}
\int_{B_{R+A_k(t)}} (\Psi(t,x) )^{p'} \mathrm{d}x = (\varrho(t))^{p'} \int_{B_{R+A_k(t)}} (\varphi(x) )^{p'} \mathrm{d}x \lesssim  (\varrho(t))^{p'} \mathrm{e}^{p'(R+A_k(t))}(R+A_k(t))^{n-1-\frac{n-1}{2}p'}.
\end{align*} Therefore, by using \eqref{asymptotic rho}, for $t\geqslant T_0$ we get
\begin{align}
\int_{B_{R+A_k(t)}} (\Psi(t,x) )^{p'} \mathrm{d}x & \lesssim  \mathrm{e}^{p'(R-\phi_k(1))} t^{\frac{k+\mu}{2}p'} (R+A_k(t))^{n-1-\frac{n-1}{2}p'} \notag \\
& \lesssim  t^{(1-k)(n-1)+\left[\frac{k+\mu}{2}-(1-k)\frac{n-1}{2}\right]p'}. \label{ub estimate Psi^p'}
\end{align} Then, combining H\"older's inequality, \eqref{lb estimate U1} and \eqref{ub estimate Psi^p'}, it follows
\begin{align}
\int_{\mathbb{R}^n} |u(t,x) |^p \mathrm{d}x & \geqslant (U_1(t))^p \left(\int_{B_{R+A_k(t)}} (\Psi(t,x) )^{p'} \mathrm{d}x \right)^{-(p-1)} \notag \\ & \gtrsim \varepsilon^p t^{kp-(1-k)(n-1)(p-1)+\left[(1-k)\frac{n-1}{2}-\frac{k+\mu}{2}\right]p}
\notag \\ & \gtrsim \varepsilon^p t^{(1-k)(n-1)+\frac{k}{2}p-\left((1-k)\frac{n-1}{2}+\frac{\mu}{2}\right)p} \label{lb estimate int |u|^p}
\end{align} for $t\geqslant T_1\doteq 2 T_0$. Finally, plugging \eqref{lb estimate int |u|^p} in \eqref{iter fram subcrit for 1st lb data}, for $t\geqslant T_1$ it holds
\begin{align*}
U_0(t) & \geqslant \int_{T_1}^t \tau^{-\mu}\int_{T_1}^\tau s^\mu  \int_{\mathbb{R}^n}  |u(s,x)|^p \, \mathrm{d} x\,  \mathrm{d}s \,  \mathrm{d}\tau  \gtrsim \varepsilon^p \int_{T_1}^t \tau^{-\mu}\int_{T_1}^\tau    s^{\mu+(1-k)(n-1)+\frac{k}{2}p-\left((1-k)\frac{n-1}{2}+\frac{\mu}{2}\right)p}  \,  \mathrm{d}s \,  \mathrm{d}\tau \\
& \gtrsim \varepsilon^p t^{-\left((1-k)\frac{n-1}{2}+\frac{\mu}{2}\right)p -\mu}  \int_{T_1}^t \int_{T_1}^\tau    (s-T_1)^{\mu+(1-k)(n-1)+\frac{k}{2}p}  \,  \mathrm{d}s \,  \mathrm{d}\tau \\
& \gtrsim \varepsilon^p  t^{-\left((1-k)\frac{n-1}{2}+\frac{\mu}{2}\right)p -\mu}      (t-T_1)^{\mu+(1-k)(n-1)+\frac{k}{2}p+2}.
\end{align*} Summarizing we proved the lower bound estimate for the functional  $U_0$
\begin{align} \label{lb estimate U0 subcrit}
U_0(t) & \geqslant K  \varepsilon^p t^{-a_0}      (t-T_1)^{b_0}
\end{align} for $t\geqslant T_1$, where $K=K(n,k,\mu,p,R,u_0,u_1)$ is a suitable positive constant and 
\begin{align}\label{def a0 b0 subcrit}
a_0\doteq \left((1-k)\tfrac{n-1}{2}+\tfrac{\mu}{2}\right)p +\mu, \qquad  b_0\doteq \mu+(1-k)(n-1)+\tfrac{k}{2}p+2.
\end{align}

\subsection{Iteration argument} \label{Subsection iter ar subcrit}

In this section we will use the iteration frame \eqref{iter fram subcrit U0} to prove that $U_0$ blows up in finite time under the assumptions of Theorem \ref{Theorem subcritical case}. More precisely, we are going to prove the sequence of lower bound estimates 
\begin{align} \label{lb estimates U0 step j subcrit}
U_0(t)  & \geqslant D_j  t^{-a_j}      (t-T_1)^{b_j}
\end{align} for $t\geqslant T_1$, where $\{D_j\}_{j\in\mathbb{N}}$, $\{a_j\}_{j\in\mathbb{N}}$ and $\{b_j\}_{j\in\mathbb{N}}$ are sequences of nonnegative real numbers that will be determined iteratively during the proof.

Clearly, for $j=0$ the estimate in  \eqref{lb estimates U0 step j subcrit} is nothing but \eqref{lb estimate U0 subcrit} with $D_0=K\varepsilon^p$ and $a_0,b_0$ defined by \eqref{def a0 b0 subcrit}. In order to prove \eqref{lb estimates U0 step j subcrit} via an inductive argument, it remains just to prove the inductive step. Let us assume the validity of \eqref{lb estimates U0 step j subcrit} for $j$. We prove now its validity for $j+1$ too.

Plugging \eqref{lb estimates U0 step j subcrit} into \eqref{iter fram subcrit U0}, for $t>T_1$ we get
\begin{align*}
U_0(t)& \geqslant  C \int_{T_1}^t \tau^{-\mu}\int_{T_1}^\tau s^{\mu -(1-k)n(p-1)}  (U_0(s))^p\,  \mathrm{d}s \,  \mathrm{d}\tau  \\
&\geqslant C D_j^p\int_{T_1}^t \tau^{-\mu}\int_{T_1}^\tau s^{\mu -(1-k)n(p-1)-a_jp}      (s-T_1)^{b_jp}\,  \mathrm{d}s \,  \mathrm{d}\tau   \\
&\geqslant C D_j^p t^{ -(1-k)n(p-1) -\mu -a_jp}   \int_{T_1}^t \int_{T_1}^\tau       (s-T_1)^{\mu+b_jp}\,  \mathrm{d}s \,  \mathrm{d}\tau  \\
& = \frac{C D_j^p}{ (1+\mu+b_j p) (2+\mu+b_j p)} \, t^{ -(1-k)n(p-1) -\mu -a_jp}     (t-T_1)^{2+\mu+b_jp},
\end{align*} which is exactly \eqref{lb estimates U0 step j subcrit} for $j+1$ provided that
\begin{align} \label{cond Dj+1}
D_{j+1} & \doteq \frac{C D_j^p}{ (1+\mu+b_j p) (2+\mu+b_j p)} ,\\
 a_{j+1} & \doteq  \underbrace{(1-k)n(p-1) + \mu}_{\doteq \alpha} + p a_j, \quad  b_{j+1}\doteq  \underbrace{2+\mu}_{\doteq \beta}+p b_j. \label{def aj+1 and bj+1}
\end{align}
Employing recursively \eqref{def aj+1 and bj+1}, we may express explicitly $a_j$ and $b_j$ as follows
\begin{align}
a_j&= \alpha+pa_{j-1} = \cdots = \alpha \sum_{k=0}^{j-1} p^k + a_0p^j =\left(\tfrac{\alpha}{p-1}+a_0\right)p^j -\tfrac{\alpha}{p-1},\label{a_j}\\
b_j&=\beta+pb_{j-1} = \cdots = \beta \sum_{k=0}^{j-1} p^k + b_0p^j =\left(\tfrac{\beta}{p-1}+b_0\right)p^j -\tfrac{\beta}{p-1}.\label{b_j}
\end{align}
Combining \eqref{def aj+1 and bj+1} and \eqref{b_j}, we find
$$b_{j}=2+\mu+pb_{j-1}<\left(\tfrac{\beta}{p-1}+b_0\right) p^j,$$ that implies, in turn,
\begin{align*}
D_{j}\geqslant \frac{CD^p_{j-1}}{(2+\mu+pb_{j-1})^2} = \frac{CD^p_{j-1}}{b_j^2} \geqslant \underbrace{\frac{C}{\left(\tfrac{\beta}{p-1}+b_0\right)^2}}_{\doteq \widetilde{C}}D^p_{j-1} p^{-2j} = \widetilde{C} D^p_{j-1} p^{-2j}.
\end{align*}
Applying the logarithmic function to both sides of the last inequality and using the resulting inequality iteratively, we get
\begin{align*}
\log D_j & \geqslant p\log D_{j-1}-2j \log p+\log \widetilde{C}\\
& \geqslant p^2\log D_{j-2}-2(j+(j-1)p)\log p+(1+p)\log \widetilde{C} \\
& \geqslant \cdots \geqslant p^{j}\log D_0-2\log p \,\sum_{k=0}^{j-1} (j-k)p^{k}+\log \widetilde{C} \, \sum_{k=0}^{j-1}p^k.
\end{align*}
Using the  formulas 
\begin{align} \label{summation identities}
\sum_{k=0}^{j-1}(j-k)p^{k}=\frac{1}{p-1}\bigg(\frac{p^{j+1}-p}{p-1}-j\bigg) \qquad \mbox{and} \qquad \sum_{k=0}^{j-1}p^k=\frac{p^j-1}{p-1},
\end{align}
that can be shown via an inductive argument, we obtain
\begin{align*}
\log D_j &  \geqslant  p^{j}\log D_0- \frac{2\log p}{p-1}\bigg(\frac{p^{j+1}-p}{p-1}-j\bigg) +(p^j-1)\frac{ \log \widetilde{C} }{p-1} \\ 
& = p^{j}\left(\log D_0 -\frac{2 p\log p}{(p-1)^2}+ \frac{ \log \widetilde{C} }{p-1}\right)+\frac{2 j\log p}{p-1}+\frac{2 p\log p}{(p-1)^2} -\frac{ \log \widetilde{C} }{p-1}.
\end{align*} Let us denote by $j_0=j_0(n,p,k,\mu)\in \mathbb{N}$ the smallest integer greater than $\frac{\log\widetilde{C}}{2\log p}-\frac{p}{p-1}$. Then, for any $j\geqslant j_0$ we have
\begin{align}
\log D_j \geqslant  p^{j}\left(\log D_0 -\frac{2 p\log p}{(p-1)^2}+ \frac{ \log \widetilde{C} }{p-1}\right) = p^{j}\log \left( K p^{-(2 p)/(p-1)^2} \widetilde{C}^{1/(p-1)} \varepsilon^p\right) = p^{j}\log \left( E_0 \varepsilon^p\right),
\label{lb Dj subcrit} 
\end{align} where $E_0\doteq K p^{-(2 p)/(p-1)^2} \widetilde{C}^{1/(p-1)}$.
Combining \eqref{lb estimates U0 step j subcrit}, \eqref{a_j}, \eqref{b_j}  and \eqref{lb Dj subcrit}, for $j\geqslant j_0$ and $t\geqslant T_1$ it holds
\begin{align*}
U_0(t) & \geqslant \exp \left(p^j \log \left( E_0 \varepsilon^p\right)\right) t^{-a_j}(t-T_1)^{b_j}  \\ &=\exp\left(p^j\left(\log \left( E_0 \varepsilon^p\right)-\left(\tfrac{\alpha}{p-1}+a_0\right)\log t+\left(\tfrac{\beta}{p-1}+b_0\right)\log (t-T_1)\right)\right) t^{\alpha/(p-1)} (t-T_1)^{-\beta/(p-1)}.
\end{align*}
For $t\geqslant 2T_1$, we have $\log(t-T_1)\geqslant \log(t/2)$, so for $j\geqslant j_0$ 
\begin{align}
U_0(t) & \geqslant \exp\left(p^j\left(\log \left( E_0 \varepsilon^p\right)+\left(\tfrac{\beta-\alpha}{p-1}+b_0-a_0\right)\log t-\left(\tfrac{\beta}{p-1}+b_0\right)\log 2\right)\right) t^{\alpha/(p-1)} (t-T_1)^{-\beta/(p-1)} \notag \\
 & = \exp\left(p^j\left(\log \left( 2^{-b_0-\beta/(p-1)}E_0 \varepsilon^p t^{\frac{\theta(n,k,\mu,p)}{p-1}}\right)\right)\right) t^{\alpha/(p-1)} (t-T_1)^{-\beta/(p-1)}, \label{final lb U0 subcrit}
\end{align} where for the exponent of $t$ in the last equality we used 
\begin{align}
\tfrac{\beta-\alpha}{p-1}+b_0-a_0 & = \tfrac{2}{p-1} -(1-k)n +(1-k)(n-1)+\tfrac{k}{2}p+2-\left((1-k)\tfrac{n-1}{2}+\tfrac{\mu}{2}\right)p \notag \\
& = \tfrac{2p}{p-1} -(1-k)-\left((1-k)\tfrac{n-1}{2}+\tfrac{\mu-k}{2}\right)p \notag \\
& = \tfrac{1}{p-1} \left\{1-k+\left((1-k)\tfrac{n+1}{2}+\tfrac{\mu+3k}{2}\right)p-\left((1-k)\tfrac{n-1}{2}+\tfrac{\mu-k}{2}\right)p^2\right\} = \tfrac{\theta(n,k,\mu,p)}{p-1}. \label{exponent t subcrit}
\end{align} Note that $\theta(n,k,\mu,p)$ is a positive quantity for $p<p_0\big(k,n+\tfrac{\mu}{1-k}\big)$.  Let us fix $\varepsilon_0>0$ sufficiently small so that
\begin{align*}
\varepsilon_0^{-\frac{p(p-1)}{\theta(n,k,\mu,p)}}\geqslant 2^{1-(b_0(p-1)+\beta)/\theta(n,k,\mu,p)} T_1.
\end{align*} Then, for any $\varepsilon\in(0,\varepsilon_0]$ and for $t\geqslant 2^{(b_0(p-1)+\beta)/\theta(n,k,\mu,p)}\varepsilon^{-\frac{p(p-1)}{\theta(n,k,\mu,p)}}$ it results
\begin{align*}
t\geqslant 2T_1 \qquad \mbox{and} \qquad 2^{-b_0-\beta/(p-1)}E_0 \varepsilon^p t^{\frac{\theta(n,k,\mu,p)}{p-1}}>1,
\end{align*} also, letting $j\to \infty$ in \eqref{final lb U0 subcrit} it turns out that $U_0(t)$ blows up. Consequently, we proved the blowing -- up of $U_0$ in finite time for any $\varepsilon\in(0,\varepsilon_0]$ whenever $p<p_0\big(k,n+\tfrac{\mu}{1-k}\big)$ and, moreover, as byproduct we found the upper bound estimate for the lifespan $T(\varepsilon)\lesssim\varepsilon^{-\frac{p(p-1)}{\theta(n,k,\mu,p)}} $ as well.

So far we applied only the lower bound estimate in \eqref{lb estimate U0 subcrit} for $U_0$. Nevertheless, we also proved another lower bound estimate for $U_0$, namely, \eqref{lb estimate U0 trivial}. Using \eqref{lb estimate U0 trivial} instead of \eqref{lb estimate U0 subcrit}, the initial values for the parameters in \eqref{lb estimates U0 step j subcrit} are $a_0=b_0=0$ and $D_0\approx \varepsilon$. Repeating the computations analogously as in the previous case and using
\begin{align*}
\log D_j \geqslant p^j  \log \left( E_1 \varepsilon\right)
\end{align*} for $j\geqslant j_1$, where $j_1$ is a suitable nonnegative integer and $E_1$ is a suitable positive constant, in place of \eqref{lb Dj subcrit} and
\begin{align*}
\tfrac{\beta-\alpha}{p-1}+b_0-a_0 & = \tfrac{2}{p-1} -(1-k)n 
\end{align*} instead of \eqref{exponent t subcrit}, we obtain immediately the blow -- up of $U_0$ in finite time for $p<p_1(k,n)$ and the corresponding upper bound estimate for the lifespan in \eqref{upper bound est lifespan subcrit thm}.

\section{Critical case: part I} \label{Section critical case p0}

In order to study the critical case $p=p_0\big(k,n+\tfrac{\mu}{1-k}\big)$, we will follow an approach which is based on the technique introduced in \cite{WakYor18} and subsequently applied to different frameworks in \cite{WakYor18Damp,PalTak19,PalTak19mix,LinTu19,ChenPal19MGT,ChenPal19SWENM,PTY20}.

From \eqref{final lb U0 subcrit} it is clear that we can no longer employ $U_0$ as functional to study the blow -- up dynamic. Therefore, we need to sharpen the choice of the functional. More precisely, we are going to consider a weighted space average of a local in time solution to \eqref{Semi EdeS k damped}. Hence, the blow -- up result will be proved by applying the so -- called \emph{slicing procedure} in an iteration argument to show a sequence of lower bound estimates for the above mentioned functional. 
Throughout this section we work under the assumptions of Theorem \ref{Theorem critical case p0}.

The section is organized as follows: in Section \ref{Subsection Aux functions} we determine a pair of auxiliary functions which have a fundamental role in the definition of the time -- dependent functional and in the determination of the iteration frame, while in Section \ref{Subsection estimates auxiliary functions} we establish some fundamental properties for these functions; finally, in Section \ref{Subsection iteration frame} we determine the iteration frame for the weighted space average whose dynamic provides the blow -- up result. 

\subsection{Auxiliary functions} \label{Subsection Aux functions}  

In this section, we introduce two auxiliary functions (see $\xi_q$ and $\eta_q$ below). These  auxiliary functions represent a generalization of the solution to the classical free wave equation given in \cite{Zhou07} and are defined by using the remarkable function  $\varphi$ introduced in \cite{YZ06},  that we have already used in the section for the subcritical case (the definition of this function is given in \eqref{def Yordanov-Zhang function}).

According to our purpose of introducing the auxiliary functions, we begin by determining the solutions $y_j=y_j(t,s;\lambda,k,\mu)$, $j\in\{0,1\}$ of the non -- autonomous, parameter -- dependent, ordinary Cauchy problems
\begin{align}\label{CP yj(t,s;lambda,k)} 
\begin{cases}  \partial_t^2 y_j(t,s;\lambda,k,\mu) - \lambda^2 t^{-2k} y_j(t,s;\lambda,k,\mu)+\mu \, t^{-1} y_j(t,s;\lambda,k,\mu) = 0, &  t>s, \\
y_j(s,s;\lambda,k,\mu)= \delta_{0j}, \\
 \partial_t y_j(s,s;\lambda,k,\mu)= \delta_{1j},
\end{cases}
\end{align} where $\delta_{ij}$ denotes the Kronecker delta, $s\geqslant 1$ is the initial time and $\lambda>0$ is a real parameter. 
To find a system of independent solutions to 
\begin{align}\label{equation y}
\frac{\mathrm{d}^2 y}{\mathrm{d} t^2} -\lambda^2 t^{-2k}y+\mu \, t^{-1} \frac{\mathrm{d} y}{\mathrm{d} t} =0
\end{align} we start by performing the change of variable $\tau= \tau(t;\lambda,k)\doteq  \lambda \phi_k(t)$. By the straightforward relations
\begin{align*}
\frac{\mathrm{d} y}{\mathrm{d} t} &= \lambda t^{-k} \frac{\mathrm{d} y}{\mathrm{d} \tau}, \qquad \frac{\mathrm{d}^2 y}{\mathrm{d} t^2} = \lambda^2 t^{-2k} \frac{\mathrm{d}^2 y}{\mathrm{d} \tau^2}-\lambda k t^{-k-1} \frac{\mathrm{d} y}{\mathrm{d} \tau},
\end{align*} it follows that $y$ solves \eqref{equation y} if and only if 
\begin{align}\label{equation y tau}
\tau \frac{\mathrm{d}^2 y}{\mathrm{d} \tau^2} +\frac{\mu-k}{1-k} \frac{\mathrm{d} y}{\mathrm{d} \tau}-\tau y=0.
\end{align} Carrying out the transformation $y(\tau)=\tau^\nu w(\tau)$ with  $\nu= \nu(k,\mu)\doteq \tfrac{1-\mu}{2(1-k)}$, it turns out that $y$ solves \eqref{equation y tau} if and only if $w$ solves the modified Bessel equation of order $\nu$
\begin{align}\label{Bessel equation w}
\tau^2\frac{\mathrm{d}^2 w}{\mathrm{d} \tau^2} +\tau\frac{\mathrm{d} w}{\mathrm{d} \tau}-\left(\nu^2+\tau^2 \right) w=0.
\end{align} 
Employing  the modified Bessel function of first and second kind of order $\nu$, denoted, respectively, by $\mathrm{I}_\nu(\tau)$ and $\mathrm{K}_\nu(\tau)$, as independent solutions to \eqref{Bessel equation w}, then, we obtain
\begin{align*}
V_0(t;\lambda,k,\mu) &\doteq \tau ^\nu \mathrm{I}_\nu (\tau) = (\lambda \phi_k(t))^\nu \mathrm{I}_\nu (\lambda \phi_k(t)), \\
V_1(t;\lambda,k,\mu) & \doteq \tau ^\nu \mathrm{K}_\nu (\tau)  = (\lambda \phi_k(t))^\nu \mathrm{K}_\nu (\lambda \phi_k(t))
\end{align*} as basis for the space of solutions to \eqref{equation y}.

\begin{proposition} \label{Proposition representations y0 and y1} The functions
\begin{align}
y_0(t,s;\lambda,k,\mu) &\doteq \lambda \, \phi_k(s) \, s^{\frac{\mu-1}{2}}  t^{\frac{1-\mu}{2}} \big[\mathrm{I}_{\nu-1}(\lambda \phi_k (s))\, \mathrm{K}_{\nu}(\lambda \phi_k (t))+\mathrm{K}_{\nu-1}(\lambda \phi_k (s))\,\mathrm{I}_{\nu}(\lambda \phi_k (t))\big],  \label{def y0(t,s;lambda,k)} \\
y_1(t,s;\lambda,k,\mu) &\doteq  (1-k)^{-1} s^{\frac{1+\mu}{2}}  t^{\frac{1-\mu}{2}} \big[\mathrm{K}_{\nu}(\lambda \phi_k (s))\, \mathrm{I}_{\nu}(\lambda \phi_k (t))-\mathrm{I}_{\nu}(\lambda \phi_k (s))\,\mathrm{K}_{\nu}(\lambda \phi_k (t))\big],  \label{def y1(t,s;lambda,k)} 
\end{align} solve the Cauchy problems \eqref{CP yj(t,s;lambda,k)} for $j=0$ and $j=1$, respectively, where $\nu= \frac{1-\mu}{2(1-k)}$ and $\mathrm{I}_\nu,\mathrm{K}_\nu$ denote the modified Bessel function of order $\nu$ of the first and second kind, respectively.
\end{proposition}

\begin{proof}
Since we proved that $V_0,V_1$ form a system of independent solutions to \eqref{equation y}, we may express the solutions to \eqref{CP yj(t,s;lambda,k)} as linear combinations of $V_0,V_1$ in the following way
\begin{align} \label{representation yj with aj and bj}
y_j(t,s;\lambda,k,\mu) = a_j(s;\lambda,k,\mu) V_0(t;\lambda,k,\mu)+  b_j(s;\lambda,k,\mu) V_1(t;\lambda,k,\mu)
\end{align} for suitable coefficients $a_j(s;\lambda,k,\mu), b_j(s;\lambda,k,\mu)$, with $j\in\{0,1\}$. 

We can describe the initial conditions $\partial^i_t y_j(s,s;\lambda,k)=\delta_{ij}$ through the system
\begin{align*}
\left(\begin{array}{cc}
V_0(s;\lambda,k,\mu) & V_1(s;\lambda,k,\mu)  \\ 
\partial_t V_0(s;\lambda,k,\mu) & \partial_t V_1(s;\lambda,k,\mu) 
\end{array} \right) \left(\begin{array}{cc}
a_0(s;\lambda,k,\mu) & a_1(s;\lambda,k,\mu)  \\ 
b_0(s;\lambda,k,\mu) & b_1(s;\lambda,k,\mu) 
\end{array} \right) = I,
\end{align*} where $I$ denotes the identity matrix. Also, to determine the coefficients in \eqref{representation yj with aj and bj}, we calculate the inverse matrix
\begin{align}\label{inverse matrix}
\left(\begin{array}{cc}
V_0(s;\lambda,k,\mu) & V_1(s;\lambda,k,\mu)  \\ 
\partial_t V_0(s;\lambda,k,\mu) & \partial_t V_1(s;\lambda,k,\mu) 
\end{array} \right)^{-1} = \left(\mathcal{W}(V_0,V_1)(s;\lambda,k,\mu)\right)^{-1}\left(\begin{array}{cc}
 \partial_t V_1(s;\lambda,k,\mu)  & -V_1(s;\lambda,k,\mu)  \\ 
-\partial_t V_0(s;\lambda,k,\mu) & V_0(s;\lambda,k,\mu) 
\end{array} \right),
\end{align} where $\mathcal{W}(V_0,V_1)$ denotes the Wronskian of $V_0,V_1$. Next, we compute explicitly the function $\mathcal{W}(V_0,V_1)$. Thanks to
\begin{align*}
\partial_t V_0(t;\lambda,k,\mu) & =\nu (\lambda\phi_k(t))^{\nu-1} \lambda \phi_k'(t) \, \mathrm{I}_\nu(\lambda \phi_k(t))+ (\lambda\phi_k(t))^{\nu} \, \mathrm{I}_\nu'(\lambda \phi_k(t))  \lambda \phi_k'(t) , \\
\partial_t V_1(t;\lambda,k,\mu) &= \nu (\lambda\phi_k(t))^{\nu-1} \lambda \phi_k'(t) \, \mathrm{K}_\nu(\lambda \phi_k(t))+ (\lambda\phi_k(t))^{\nu} \, \mathrm{K}_\nu'(\lambda \phi_k(t))  \lambda \phi_k'(t),
\end{align*} recalling $\phi'_k(t)=t^{-k}$ and $2\nu-1=\frac{k-\mu}{1-k}$, we can express $\mathcal{W}(V_0,V_1)$ as follows:
\begin{align*}
\mathcal{W}(V_0,V_1)(t;\lambda,k,\mu) & = (\lambda \phi_k(t))^{2\nu} (\lambda \phi_k'(t)) \big\{\mathrm{K}_\nu'(\lambda \phi_k(t))\, \mathrm{I}_\nu(\lambda \phi_k(t)) -\mathrm{I}_\nu'(\lambda \phi_k(t))\, \mathrm{K}_\nu(\lambda \phi_k(t)) \big\} \\
& = (\lambda \phi_k(t))^{2\nu} (\lambda \phi_k'(t)) \mathcal{W}(\mathrm{I}_\nu,\mathrm{K}_\nu) (\lambda\phi_k(t)) = -(\lambda \phi_k(t))^{2\nu-1} (\lambda \phi_k'(t))\\
& = -\lambda^{2\nu}  (\phi_k(t))^{2\nu-1}  \phi_k'(t) = -c_{k,\mu}^{-1} \lambda^{2\nu} t^{-\mu},
\end{align*} where $c_{k,\mu}\doteq (1-k)^{\frac{k-\mu}{1-k}}$ and in the third equality we used the value of the Wronskian of $\mathrm{I}_\nu,\mathrm{K}_\nu$
\begin{align*}
\mathcal{W}(\mathrm{I}_\nu,\mathrm{K}_\nu) (z)= \mathrm{I}_\nu (z) \frac{\partial \,\mathrm{K}_\nu}{\partial z}(z)- \mathrm{K}_\nu (z)  \frac{\partial \,\mathrm{I}_\nu}{\partial z}(z)=- \frac1z.
\end{align*} 
 Plugging the previously determined representation of  $\mathcal{W}(V_0,V_1)$ in \eqref{inverse matrix}, we have
\begin{align*}
\left(\begin{array}{cc}
a_0(s;\lambda,k,\mu) & a_1(s;\lambda,k,\mu)  \\ 
b_0(s;\lambda,k,\mu) & b_1(s;\lambda,k,\mu) 
\end{array} \right) = c_{k,\mu} \lambda^{-2\nu} s^{\mu} \left(\begin{array}{cc}
- \partial_t V_1(s;\lambda,k,\mu)  & V_1(s;\lambda,k,\mu)  \\ 
\partial_t V_0(s;\lambda,k,\mu) & -V_0(s;\lambda,k,\mu) 
\end{array} \right).
\end{align*}
Let us begin by showing \eqref{def y0(t,s;lambda,k)}. Using the above representation of $a_0(s;\lambda,k\mu),b_0(s;\lambda,k,\mu)$ in \eqref{representation yj with aj and bj}, we find
\begin{align*}
y_0(t,s;\lambda,k,\mu) &= c_{k,\mu} \lambda^{-2\nu} s^{\mu} \big\{\partial_tV_0(s;\lambda,k,\mu) V_1(t;\lambda,k,\mu)-\partial_tV_1(s;\lambda,k,\mu) V_0(t;\lambda,k,\mu)\big\} \\
& = c_{k,\mu} \, \nu \, s^{\mu} \phi_k'(s) (\phi_k(s))^{\nu-1}(\phi_k(t))^{\nu}   \big\{ \mathrm{I}_\nu(\lambda \phi_k(s))\,  \mathrm{K}_\nu(\lambda \phi_k(t)) -\mathrm{K}_\nu(\lambda \phi_k(s))\,  \mathrm{I}_\nu(\lambda \phi_k(t)) \big\} \\
& \quad + c_{k,\mu} \,\lambda \, s^{\mu} \phi_k'(s) (\phi_k(s))^{\nu} (\phi_k(t))^{\nu}  \big\{ \mathrm{I}'_\nu(\lambda \phi_k(s))\,  \mathrm{K}_\nu(\lambda \phi_k(t)) -\mathrm{K}'_\nu(\lambda \phi_k(s))\,  \mathrm{I}_\nu(\lambda \phi_k(t)) \big\}.
\end{align*}
Using the following recursive relations for the derivatives of the modified Bessel functions
\begin{align*}
\frac{\partial \,\mathrm{I}_\nu}{\partial z}(z) & = -\frac{ \nu}{z} \, \mathrm{I}_\nu(z)+ \mathrm{I}_{\nu-1}(z), \\
\frac{\partial  \,\mathrm{K}_\nu}{\partial z}(z) & = -\frac{ \nu}{z} \, \mathrm{K}_\nu(z)- \mathrm{K}_{\nu-1}(z),
\end{align*} there is a cancellation in the last relation, so, we arrive at
\begin{align}
y_0(t,s;\lambda,k,\mu) & = c_{k,\mu} \, \lambda \, s^\mu \phi_k'(s) (\phi_k(s)\phi_k(t))^{\nu}  \big\{ \mathrm{I}_{\nu-1}(\lambda \phi_k(s))\,  \mathrm{K}_\nu(\lambda \phi_k(t)) +\mathrm{K}_{\nu-1}(\lambda \phi_k(s))\,  \mathrm{I}_\nu(\lambda \phi_k(t)) \big\}. \label{intermediate representation y0}
\end{align} Thanks to $$c_{k,\mu} \, s^\mu \phi_k'(s) (\phi_k(s)\phi_k(t))^{\nu} = (1-k)^{-1} s^{\mu-k }(st)^{\frac{1-\mu}{2}}  =\phi_k(s) s^{\frac{\mu-1}{2}} t^{\frac{1-\mu}{2}},$$ from  \eqref{intermediate representation y0} it follows immediately \eqref{def y0(t,s;lambda,k)}. Let us show now the representation for $y_1$. Plugging the above determined expressions for $a_1(s;\lambda,k,\mu),b_1(s;\lambda,k,\mu)$ in \eqref{representation yj with aj and bj}, we get
\begin{align}
y_1(t,s;\lambda,k,\mu) &= c_{k,\mu} \lambda^{-2\nu} s^\mu \big\{V_1(s;\lambda,k,\mu) V_0(t;\lambda,k,\mu)-V_0(s;\lambda,k,\mu) V_1(t;\lambda,k,\mu)\big\} \notag \\
& = c_{k,\mu} \lambda^{-2\nu} s^\mu (\lambda\phi_k(s))^{\nu} (\lambda\phi_k(t))^{\nu}\big\{  \mathrm{K}_\nu(\lambda \phi_k(s))   \, \mathrm{I}_\nu(\lambda \phi_k(t)) -\mathrm{I}_\nu(\lambda \phi_k(s)) \, \mathrm{K}_\nu(\lambda \phi_k(t))\big\} \notag \\
& = c_{k,\mu} s^\mu (\phi_k(s) \phi_k(t))^{\nu}\big\{  \mathrm{K}_\nu(\lambda \phi_k(s))   \, \mathrm{I}_\nu(\lambda \phi_k(t)) -\mathrm{I}_\nu(\lambda \phi_k(s)) \, \mathrm{K}_\nu(\lambda \phi_k(t))\big\} . \label{intermediate representation y1}
\end{align} Hence, due to $c_{k,\mu} s^\mu (\phi_k(s) \phi_k(t))^{\nu} = (1-k)^{-1} s^{\frac{1+\mu}{2}}  t^{\frac{1-\mu}{2}}$, from \eqref{intermediate representation y1} it results \eqref{def y1(t,s;lambda,k)}. The proof is complete.
\end{proof}

\begin{lemma} Let $y_0$, $y_1$ be the functions defined in \eqref{def y0(t,s;lambda,k)} and \eqref{def y1(t,s;lambda,k)}, respectively. Then, the following identities are satisfied for any $t\geqslant s\geqslant 1$
\begin{align}
& \frac{\partial y_1}{\partial s}(t,s;\lambda,k,\mu)= -y_0(t,s;\lambda,k,\mu)+\mu s^{-1} y_1(t,s;\lambda,k,\mu), \label{dy1/ds= -y0} \\
& \frac{\partial^2 y_1}{\partial s^2}(t,s;\lambda,k,\mu) -\lambda^2 s^{-2k}y_1(t,s;\lambda,k,\mu)-\mu s^{-1} \frac{\partial y_1}{\partial s}(t,s;\lambda,k,\mu)+\mu s^{-2} y_1(t,s;\lambda,k,\mu)= 0. \label{y1 adjoiunt equation} 
\end{align} 
\end{lemma}

\begin{remark} As the operator $\partial_s^2 -\lambda^2s^{-2k}-\mu s^{-1} \partial_s+\mu s^{-2}$ is the formal adjoint of $\partial_t^2 -\lambda^2t^{-2k}+\mu t^{-1} \partial_t$, in particular, \eqref{dy1/ds= -y0} and \eqref{y1 adjoiunt equation} tell us that $y_1$ solves also the adjoint problem to \eqref{equation y} with final conditions $(0,-1)$. 
\end{remark}

\begin{proof} Let us introduce the pair of independent solutions to \eqref{equation y}
\begin{align*}
 z_0(t;\lambda,k,\mu) & \doteq y_0(t,1;\lambda,k,\mu), \\
 z_1(t;\lambda,k,\mu) & \doteq y_1(t,1;\lambda,k,\mu).
\end{align*} Since the Wronskian $\mathcal{W}(z_0,z_1)(t;\lambda,k,\mu)$ solves the differential equation $\mathcal{W}'(z_0,z_1)=-\mu t^{-1} \mathcal{W}(z_0,z_1)$ with initial condition $\mathcal{W}(z_0,z_1)(1;\lambda,k,\mu)=1$, then, $\mathcal{W}(z_0,z_1)(t;\lambda,k,\mu)=t^{-\mu}$. Therefore,  repeating similar computations as in the proof of Proposition \ref{Proposition representations y0 and y1}, we may show the representations
\begin{align*}
y_0(t,s;\lambda,k,\mu) & = s^\mu \left\{ z_1'(s;\lambda,k,\mu)  z_0(t;\lambda,k,\mu)- z_0'(s;\lambda,k,\mu)  z_1(t;\lambda,k,\mu)\right\},\\
y_1(t,s;\lambda,k,\mu) & =  s^\mu \left\{ z_0(s;\lambda,k,\mu)  z_1(t;\lambda,k,\mu)- z_1(s;\lambda,k,\mu)  z_0(t;\lambda,k,\mu)\right\}.
\end{align*}  Let us prove \eqref{dy1/ds= -y0}. Differentiating the second one of the previous representations with respect to $s$, we find
\begin{align*}
\frac{\partial y_1}{\partial s}(t,s;\lambda,k) & = \mu s^{\mu-1} \left\{ z_0(s;\lambda,k,\mu)  z_1(t;\lambda,k,\mu)- z_1(s;\lambda,k,\mu)  z_0(t;\lambda,k,\mu)\right\} \\ & \quad + s^\mu \left\{ z_0'(s;\lambda,k,\mu)  z_1(t;\lambda,k,\mu)- z_1'(s;\lambda,k,\mu)  z_0(t;\lambda,k,\mu)\right\} \\ &= \mu s^{-1} y_1(t,s;\lambda,k,\mu)-y_0(t,s;\lambda,k,\mu).
\end{align*} On the other hand, due to the fact that $z_0,z_1$ satisfy \eqref{equation y}, then,
\begin{align*}
 \frac{\partial^2 y_1}{\partial s^2}(t,s;\lambda,k) &  = s^\mu \left\{ z_0''(s;\lambda,k,\mu)  z_1(t;\lambda,k,\mu)- z_1''(s;\lambda,k,\mu)  z_0(t;\lambda,k,\mu)\right\} \\
 & \quad + 2 \mu s^{\mu-1} \left\{ z_0'(s;\lambda,k,\mu)  z_1(t;\lambda,k,\mu)- z_1'(s;\lambda,k,\mu)  z_0(t;\lambda,k,\mu)\right\} \\
 & \quad + \mu(\mu-1) s^{\mu-2} \left\{ z_0(s;\lambda,k,\mu)  z_1(t;\lambda,k,\mu)- z_1(s;\lambda,k,\mu)  z_0(t;\lambda,k,\mu)\right\} \\
 &= s^\mu \big\{ \left[\lambda^2 s^{-2k}z_0(s;\lambda,k,\mu)-\mu s^{-1}z_0'(s;\lambda,k,\mu)\right]  z_1(t;\lambda,k,\mu) \\ & \qquad \quad - \left[\lambda^2 s^{-2k}z_1(s;\lambda,k,\mu)-\mu s^{-1}z_1'(s;\lambda,k,\mu)\right] z_0(t;\lambda,k,\mu)\big\} \\
 & \quad + 2 \mu s^{\mu-1} \left\{ z_0'(s;\lambda,k,\mu)  z_1(t;\lambda,k,\mu)- z_1'(s;\lambda,k,\mu)  z_0(t;\lambda,k,\mu)\right\} \\
 & \quad + \mu(\mu-1) s^{\mu-2} \left\{ z_0(s;\lambda,k,\mu)  z_1(t;\lambda,k,\mu)- z_1(s;\lambda,k,\mu)  z_0(t;\lambda,k,\mu)\right\} \\
 &=  \lambda^2 s^{-2k} s^\mu \big\{ z_0(s;\lambda,k,\mu)  z_1(t;\lambda,k,\mu)  - z_1(s;\lambda,k,\mu) z_0(t;\lambda,k,\mu)\big\} \\
 & \quad +  \mu s^{\mu-1} \left\{ z_0'(s;\lambda,k,\mu)  z_1(t;\lambda,k,\mu)- z_1'(s;\lambda,k,\mu)  z_0(t;\lambda,k,\mu)\right\} \\
 & \quad + \mu(\mu-1) s^{\mu-2} \left\{ z_0(s;\lambda,k,\mu)  z_1(t;\lambda,k,\mu)- z_1(s;\lambda,k,\mu)  z_0(t;\lambda,k,\mu)\right\} \\
 &= \lambda^2 s^{-2k} y_1(t,s;\lambda,k,\mu) -\mu s^{-1} y_0(t,s;\lambda,k,\mu) +\mu(\mu-1) s^{-2} y_1(t,s;\lambda,k,\mu).
 \end{align*} Applying \eqref{dy1/ds= -y0}, from the last chain of equalities we get
 \begin{align*}
\frac{\partial^2 y_1}{\partial s^2}(t,s;\lambda,k)  &= \lambda^2 s^{-2k} y_1(t,s;\lambda,k,\mu) +\mu s^{-1} \left(\frac{\partial y_1}{\partial s}(t,s;\lambda,k)-\mu s^{-1} y_1(t,s;\lambda,k,\mu) \right) \\ & \quad +\mu(\mu-1) s^{-2} y_1(t,s;\lambda,k,\mu) \\
&= \lambda^2 s^{-2k} y_1(t,s;\lambda,k,\mu) +\mu s^{-1} \frac{\partial y_1}{\partial s}(t,s;\lambda,k)  -\mu s^{-2} y_1(t,s;\lambda,k,\mu).
\end{align*} Thus, we proved \eqref{y1 adjoiunt equation} too. This completes the proof.
\end{proof}

\begin{proposition} \label{Proposition integral relation test function} Let $u_0\in H^1(\mathbb{R}^n)$ and $u_1\in L^2(\mathbb{R}^n)$ be functions such that $\mathrm{supp}\, u_j \subset B_R$ for $j=0,1$ and for some $R>0$ and let $\lambda>0$ be a parameter. Let $u$ be a local in time energy solution to \eqref{Semi EdeS k damped} on $[1,T)$ according to Definition \ref{Def energy sol}. Then, the following integral identity is satisfied for any $t\in [1,T)$
\begin{align} 
\int_{\mathbb{R}^n} u(t,x) \varphi_\lambda (x) \, \mathrm{d}x & = \varepsilon \, y_0(t,1;\lambda,k) \int_{\mathbb{R}^n} u_0(x) \varphi_\lambda(x) \, \mathrm{d}x + \varepsilon \, y_1(t,1;\lambda,k) \int_{\mathbb{R}^n} u_1(x) \varphi_\lambda(x) \, \mathrm{d}x \notag \\ & \quad +\int_1^t  y_1(t,s;\lambda,k)  \int_{\mathbb{R}^n} |u(s,x)|^p \varphi_\lambda (x) \, \mathrm{d}x \, \mathrm{d}s, \label{fundametal integral equality}
\end{align} where $\varphi_\lambda(x)\doteq \varphi(\lambda x)$ and $\varphi$ is defined by \eqref{def Yordanov-Zhang function}.
\end{proposition}

\begin{proof} Assuming $u_0,u_1$ compactly supported, we can consider a test function $\psi\in \mathcal{C}^\infty([1,T)\times\mathbb{R}^n)$ in Definition \ref{Def energy sol} according to Remark \ref{Remark support}. Hence, we take $\psi(s,x)=y_1(t,s;\lambda,k,\mu)\varphi_\lambda(x)$ (here $t,\lambda$ can be treated as fixed parameters). Consequently, $\psi$ satisfies
\begin{align*}
\psi(t,x) &=y_1(t,t;\lambda,k,\mu) \varphi_\lambda(x)=0, \qquad   
\psi(1,x) =y_1(t,1;\lambda,k,\mu) \varphi_\lambda(x), \\
\psi_s(t,x) &= \partial_s y_1(t,t;\lambda,k,\mu) \varphi_\lambda(x) =\left(\mu t^{-1}  y_1(t,t;\lambda,k,\mu)-  y_0(t,t;\lambda,k,\mu)\right) \varphi_\lambda(x) =- \varphi_\lambda(x), \\
 \psi_s(1,x) &=  \partial_s y_1(t,1;\lambda,k,\mu) \varphi_\lambda(x) =\left( \mu y_1(t,1;\lambda,k,\mu)-  y_0(t,1;\lambda,k,\mu)\right)\varphi_\lambda(x), 
\end{align*}  and
\begin{align*}
\psi_{ss}(s,x) -s^{-2k} \Delta \psi(s,x) - \mu\partial_s(s^{-1}  \psi(s,x)) =\left( \partial_s^2 -\lambda^2 s^{-2k} -\mu s^{-1}\partial_s+\mu s^{-2}\right)  y_1(t,s;\lambda,k,\mu)\varphi_\lambda(x)=0,
\end{align*}
where we used \eqref{dy1/ds= -y0}, \eqref{y1 adjoiunt equation} and the property $\Delta \varphi_\lambda=\lambda^2\varphi_\lambda$. 
Then, employing the above defined $\psi$ in \eqref{integral identity  weak sol}, we find immediately \eqref{fundametal integral equality}. This completes the proof.
\end{proof}

\begin{proposition} \label{Proposition lower bound estimates y0 and y1} Let $y_0$, $y_1$ be the functions defined in \eqref{def y0(t,s;lambda,k)} and \eqref{def y1(t,s;lambda,k)}, respectively. Then, the following estimates are satisfied for any $t\geqslant s\geqslant 1$
\begin{align}
& y_0(t,s;\lambda,k,\mu)\geqslant s^{\frac{\mu-k}{2}} t^{\frac{k-\mu}{2}} \cosh \big(\lambda (\phi_k(t)-\phi_k(s))\big)  \qquad \mbox{if} \ \mu\in[2-k,\infty), \label{lower bound estimate y0(t,s;lambda,k)} \\
& y_1(t,s;\lambda,k,\mu)\geqslant  s^{\frac{\mu+k}{2}} t^{\frac{k-\mu}{2}} \, \frac{\sinh \big(\lambda (\phi_k(t)-\phi_k(s))\big) }{\lambda}  \qquad \mbox{if} \  \mu\in[0, k]\cup[2-k,\infty).\label{lower bound estimate y_1(t,s;lambda,k)}
\end{align} 
\end{proposition}

\begin{proof} The proof of the inequalities \eqref{lower bound estimate y0(t,s;lambda,k)} and \eqref{lower bound estimate y_1(t,s;lambda,k)} is based on the following minimum type principle: \\
\emph{ let $w=w(t,s;\lambda,k,\mu)$ be a solution of the Cauchy problem }
\begin{align}\label{CP y}
\begin{cases} \partial_t^2 w -\lambda^2 t^{-2k} w+\mu\,  t^{-1} \partial_t w=h, &  \mbox{for} \ t>s \geqslant 1, \\  w(s)=
w_0, \ \partial_t w(s)=w_1, \end{cases}
\end{align} \emph{where $h=h(t,s;\lambda,k,\mu)$ is a continuous function; if $h\geqslant 0$ and $w_0=w_1=0$ (i.e. $w$ is a \emph{supersolution} of the homogeneous problem with trivial initial conditions), then, $w(t,s;\lambda,k,\mu)\geqslant 0$ for any $t>s$}.

 In order to prove this minimum principle, we apply the continuous dependence on initial conditions (note that for $t\geqslant 1$ the functions $t^{-2k}$ and $\mu t^{-1}$ are smooth). Indeed, if we denote by $w_\epsilon$ the solution to \eqref{CP y} with $w_0=\epsilon>0$ and $w_1=0$, then, $w_\epsilon$ solves the integral equation
\begin{align*}
w_\epsilon(t,s;\lambda,k,\mu) = \epsilon +\int_s^t \tau^{-\mu}\int_s^\tau \sigma^\mu\big(\lambda^2\sigma^{-2k} w_\epsilon(\sigma,s;\lambda,k,\mu)+h(\sigma,s;\lambda,k,\mu)\big) \mathrm{d}\sigma\, \mathrm{d}\tau.
\end{align*} By contradiction, one can prove easily that $w_\epsilon(t,s;\lambda,k,\mu)>0$ for any $t>s$. Hence, by the continuous dependence on initial data, letting $\epsilon\to 0$, we find that  $w(t,s;\lambda,k,\mu)\geqslant 0$ for any $t>s$.


Let us prove the validity of \eqref{lower bound estimate y_1(t,s;lambda,k)}. Denoting by $w_1=w_1(t,s;\lambda,k,\mu)$ the function on the right -- hand side of \eqref{lower bound estimate y_1(t,s;lambda,k)}, we find immediately $w_1(s,s;\lambda,k,\mu)=0$ and $\partial_t w_1(s,s;\lambda,k,\mu)=1$. Moreover,
\begin{align*}
\partial_t^2 w_1(t,s;\lambda,k,\mu) &= \lambda^{-1} s^{\frac{k+\mu}{2}}  t^{\frac{k-\mu}{2}}  \Big[ \tfrac{k-\mu}{2} \left(\tfrac{ k-\mu}{2} -1\right) t^{-2} \sinh \big(\lambda (\phi_k(t)-\phi_k(s))\big)\\
 & \ \quad \phantom{ \lambda^{-1} s^{\frac{k+\mu}{2}}  t^{\frac{k-\mu}{2}}  \Big[ }+ (k-\mu) \, t^{-1 } \cosh \big(\lambda (\phi_k(t)-\phi_k(s))\big) \lambda \phi'_k(t) \\ 
  &  \ \quad \phantom{ \lambda^{-1} s^{\frac{k+\mu}{2}}  t^{\frac{k-\mu}{2}}  \Big[ } + \sinh \big(\lambda (\phi_k(t)-\phi_k(s))\big) (\lambda \phi'_k(t))^2+\cosh \big(\lambda (\phi_k(t)-\phi_k(s))\big) \lambda \phi''_k(t)  \Big] \\
  &= \left[\tfrac{k-\mu}{2} \left(\tfrac{ k-\mu}{2} -1\right) t^{-2}  +\lambda^2 t^{-2k} \right] w_1(t,s;\lambda,k,\mu) - \mu  s^{\frac{k+\mu}{2}}  t^{-1-\frac{k+\mu}{2}} \cosh \big(\lambda (\phi_k(t)-\phi_k(s))\big)  
\end{align*} and
\begin{align*}
\partial_t w_1(t,s;\lambda,k,\mu) &= \lambda^{-1} s^{\frac{k+\mu}{2}}  t^{\frac{k-\mu}{2}}  \Big[ \tfrac{k-\mu}{2} \,  t^{-1} \sinh \big(\lambda (\phi_k(t)-\phi_k(s))\big) + \lambda t^{-k} \cosh \big(\lambda (\phi_k(t)-\phi_k(s))\big)  \Big] \\
&=  \tfrac{k-\mu}{2} \,  t^{-1} w_1(t,s;\lambda,k,\mu)  + s^{\frac{k+\mu}{2}}  t^{-\frac{k+\mu}{2}} \cosh \big(\lambda (\phi_k(t)-\phi_k(s))\big)  
\end{align*} imply that
\begin{align*}
 \partial_t^2 w_1(t,s;\lambda,k,\mu) - \lambda^2 t^{-2k}  w_1(t,s;\lambda,k,\mu)+\mu \, t^{-1} \partial_t w_1(t,s;\lambda,k,\mu)  = \tfrac{k-\mu}{2} \left(\tfrac{ k+\mu}{2}-1\right) w_1(t,s;\lambda,k,\mu) \leqslant 0 ,
\end{align*} where in the last step we employ the assumption $\mu\notin (k,2-k)$ to guarantee that the multiplicative constant is negative.
Therefore, $y_1-w_1$ is a supersolution of \eqref{CP y} with $h=0$ and $w_0=w_1=0$. Thus, applying the minimum principle we have that $(y_1-w_1)(t,s;\lambda,k)\geqslant 0$  for any $t>s$, that is, we showed \eqref{lower bound estimate y_1(t,s;lambda,k)}.

 In a completely analogous way, one can prove \eqref{lower bound estimate y0(t,s;lambda,k)}, repeating the previous argument based on the minimum principle with  $w_0(t,s;\lambda,k,\mu)\doteq s^{\frac{\mu-k}{2}} t^{\frac{k-\mu}{2}}   \cosh \big(\lambda (\phi_k(t)-\phi_k(s))\big)$ in place of $w_1(t,s;\lambda,k,\mu)$ and $y_0$ in place of $y_1$, respectively. However, in order to guarantee that $w_0(s,s;\lambda,k,\mu)=1$ and $\partial_t w_0(s,s;\lambda,k,\mu)\leqslant 0$, we are forced to require $\mu\geqslant k$, which provides, together with the condition $\mu\notin (k,2-k)$ that is necessary to ensure that $w_0$ is actually a subsolution of the homogeneous equation, the range for $\mu$ in \eqref{lower bound estimate y0(t,s;lambda,k)}.
\end{proof}

\begin{remark} \label{Remark transf v}
Although \eqref{lower bound estimate y0(t,s;lambda,k)} might be restrictive from the viewpoint of the range for $\mu$ in the statement of Theorem \ref{Theorem critical case p0}, we can actually overcome this difficulty by showing a transformation which allows to link the case $\mu\in [0,k]$ to the case $\mu \in [2-k,2]$, when a lower bound estimate for $y_0$ is available. Indeed, if we perform the transformation  $v=v(t,x) \doteq t^{\mu-1} u(t,x)$, then, $u$ is a solution to \eqref{Semi EdeS k damped} if and only if $v$ solves 
\begin{align}\label{Semi EdeS k damped transf v} 
\begin{cases}  v_{tt} - t^{-2k}\Delta v+(2-\mu) \, t^{-1} v_t=t^{(1-\mu)(p-1)} |v|^p & x\in \mathbb{R}^n, \ t\in (1,T), \\
v(1,x)= \varepsilon u_0(x) & x\in \mathbb{R}^n, \\
 u_t(1,x)= \varepsilon u_1(x)+ \varepsilon (1-\mu) u_0(x) & x\in \mathbb{R}^n.
\end{cases}
\end{align} Let us point out that in \eqref{Semi EdeS k damped transf v} a time -- dependent factor which decays with polynomial order appears in the nonlinear term on the right -- hand side. Therefore, we will reduce the case $\mu\in [0,k]$ to the case $\mu\geqslant 2-k$, up to the time -- dependent factor $t^{(1-\mu)(p-1)} $ in the nonlinearity.
\end{remark}

We can introduce now for $t\geqslant s \geqslant 1$ and $x\in\mathbb{R}^n$ the definition of the following \emph{auxiliary function}
\begin{align}
\xi_q(t,s,x;k,\mu) &\doteq \int_0^{\lambda_0} \mathrm{e}^{-\lambda (A_k(t)+R)} y_0(t,s;\lambda,k,\mu) \, \varphi_\lambda(x) \lambda^q \,\mathrm{d}\lambda \label{def xi q}, \\
\eta_q(t,s,x;k,\mu) & \doteq \int_0^{\lambda_0} \mathrm{e}^{-\lambda (A_k(t)+R)} \frac{y_1(t,s;\lambda,k,\mu)}{ \phi_k(t)-\phi_k(s) }\, \varphi_\lambda(x) \lambda^q \,\mathrm{d}\lambda\label{def eta q},
\end{align}  where $q>-1$, $\lambda_0>0$ is a fixed parameter and $A_k$ is defined by \eqref{def A k}.

Combining Proposition \ref{Proposition integral relation test function} and \eqref{def xi q} and \eqref{def eta q}, we establish a fundamental equality, whose role will be crucial in the next sections in order to prove the blow -- up result.

\begin{corollary} \label{Corollary fund ineq}  Let $u_0\in H^1(\mathbb{R}^n)$ and $u_1\in L^2(\mathbb{R}^n)$ such that $\mathrm{supp}\, u_j \subset B_R$ for $j=0,1$ and for some $R>0$. Let $u$ be a local in time energy solution to \eqref{Semi EdeS k damped} on $[1,T)$ according to Definition \ref{Def energy sol}. Let $q>-1$ and let $\xi_q(t,s,x;k),\eta_q(t,s,x;k)$ be the functions defined by \eqref{def xi q} and \eqref{def eta q}, respectively. Then,
\begin{align}
\int_{\mathbb{R}^n} u(t,x) \, \xi_q(t,t,x;k,\mu) \, \mathrm{d}x & = \varepsilon \int_{\mathbb{R}^n} u_0(x) \,  \xi_q(t,1,x;k,\mu)  \, \mathrm{d}x 
+ \varepsilon \,  (\phi_k(t)-\phi_k(1)) \int_{\mathbb{R}^n} u_1(x) \, \eta_q(t,s,x;k,\mu)  \, \mathrm{d}x \notag \\ & \quad +\int_1^t  (\phi_k(t)-\phi_k(s))   \int_{\mathbb{R}^n} |u(s,x)|^p \eta_q(t,s,x;k,\mu) \, \mathrm{d}x \, \mathrm{d}s \label{fundamental inequality functional mathcalU}
\end{align} for any $t\in [1,T)$.
\end{corollary}

\begin{proof}
 Multiplying both sides of \eqref{fundametal integral equality}  by $\mathrm{e}^{-\lambda (A_k(t)+R)}\lambda^q$, integrating with respect to $\lambda$ over $[0,\lambda_0]$ and applying Fubini's theorem, we get easily \eqref{fundamental inequality functional mathcalU}.
\end{proof}

\subsection{Properties of the auxiliary functions} \label{Subsection estimates auxiliary functions}

In this section, we establish lower and upper bound estimates for the auxiliary functions $\xi_q,\eta_q$ under suitable assumptions on $q$. In the lower bound estimates, we may restrict our considerations to the case $\mu\geqslant 2-k$ thanks to Remark \ref{Remark transf v}, even though the  estimate for $\eta_q$  that will be proved thanks to \eqref{lower bound estimate y_1(t,s;lambda,k)} clearly would be true also for $\mu\in [0,k]$.

\begin{lemma} \label{Lemma lower bound estimates xi q and eta q} Let $n\geqslant 1$, $k\in[0,1)$, $\mu\geqslant 2-k$ and $\lambda_0>0$. If we assume  $q>-1$, then, for $t\geqslant s\geqslant 1$ and $|x|\leqslant A_k(s) +R$ the following lower bound estimates are satisfied:
\begin{align}
\xi_q (t,s,x;k,\mu) &\geqslant B_0 s^{\frac{\mu-k}{2}} t^{\frac{k-\mu}{2}} \langle A_k(s)\rangle ^{-q-1}; \label{lower bound xi q}\\
\eta_q (t,s,x;k,\mu) & \geqslant B_1 s^{\frac{\mu+k}{2}} t^{\frac{k-\mu}{2}}  \langle A_k(t)\rangle ^{-1}\langle A_k(s)\rangle ^{-q}. \label{lower bound eta q}
\end{align}
Here $B_0,B_1$ are positive constants depending only on $\lambda_0,q,R,k$ and we employ the notation $\langle y\rangle \doteq 3+|s|$.
\end{lemma}

\begin{proof} We adapt the main ideas in the proof of Lemma 3.1 in \cite{WakYor18} to our model. Since 
\begin{align} \label{asymptotic YordanovZhang function}
\langle |x| \rangle^{-\frac{n-1}{2}}\mathrm{e}^{|x|}\lesssim \varphi(x) \lesssim\langle |x| \rangle^{-\frac{n-1}{2}}\mathrm{e}^{|x|} 
\end{align} holds for any $x\in \mathbb{R}^n$, there exists  a constant $B=B(\lambda_0,R,k)>0$ independent of $\lambda$ and $s$ such that
$$B\leqslant \inf_{\lambda \in \left[\frac{\lambda_0}{\langle A_k(s)\rangle}, \frac{2\lambda_0}{\langle A_k(s)\rangle}\right]} \inf_{|x|\leqslant A_k(s)+R} \mathrm{e}^{-\lambda(A_k(s)+R)}\varphi_\lambda(x).$$
Let us begin by proving \eqref{lower bound xi q}. Using the lower bound estimate in \eqref{lower bound estimate y0(t,s;lambda,k)}, shrinking the domain of integration in \eqref{def xi q} to $\left[\frac{\lambda_0}{\langle A_k(s)\rangle}, \frac{2\lambda_0}{\langle A_k(s)\rangle}\right]$ and applying the previous inequality, we arrive at
\begin{align*}
\xi_q (t,s,x;k,\mu) &\geqslant s^{\frac{\mu-k}{2}} t^{\frac{k-\mu}{2}} \int_{\lambda_0/\langle A_k(s)\rangle} ^{2\lambda_0/\langle A_k(s)\rangle} \mathrm{e}^{-\lambda (A_k(t)-A_k(s))} \cosh \big(\lambda (\phi_k(t)-\phi_k(s))\big) \mathrm{e}^{-\lambda (A_k(s)+R)}  \varphi_\lambda(x) \lambda^q \,\mathrm{d}\lambda \\
&\geqslant B s^{\frac{\mu-k}{2}} t^{\frac{k-\mu}{2}} \int_{\lambda_0/\langle A_k(s)\rangle} ^{2\lambda_0/\langle A_k(s)\rangle} \mathrm{e}^{-\lambda (A_k(t)-A_k(s))} \cosh \big(\lambda (\phi_k(t)-\phi_k(s))\big)  \lambda^q \,\mathrm{d}\lambda \\
& =  \tfrac{B}{2} s^{\frac{\mu-k}{2}} t^{\frac{k-\mu}{2}} \int_{\lambda_0/\langle A_k(s)\rangle} ^{2\lambda_0/\langle A_k(s)\rangle} \left(1+\mathrm{e}^{-2\lambda (\phi_k(t)-\phi_k(s))}\right)  \lambda^q \,\mathrm{d}\lambda \\ & \geqslant  \tfrac{B}{2} s^{\frac{\mu-k}{2}} t^{\frac{k-\mu}{2}} \int_{\lambda_0/\langle A_k(s)\rangle} ^{2\lambda_0/\langle A_k(s)\rangle} \lambda^q \,\mathrm{d}\lambda =  \frac{B(2^{q+1}-1) \lambda_0^{q+1}}{2(q+1)}  \langle A_k(s)\rangle^{-q-1}.
\end{align*}  Repeating similar steps as before, thanks to \eqref{lower bound estimate y_1(t,s;lambda,k)} we obtain
\begin{align*}
\eta_q (t,s,x;k,\mu) &\geqslant s^{\frac{\mu+k}{2}} t^{\frac{k-\mu}{2}} \int_{\lambda_0/\langle A_k(s)\rangle} ^{2\lambda_0/\langle A_k(s)\rangle} \mathrm{e}^{-\lambda (A_k(t)-A_k(s))} \frac{\sinh \big(\lambda (\phi_k(t)-\phi_k(s))\big) }{\lambda  (\phi_k(t)-\phi_k(s)) }\, \mathrm{e}^{-\lambda (A_k(s)+R)}  \varphi_\lambda(x) \lambda^q \,\mathrm{d}\lambda \\
&\geqslant \tfrac{B}{2} s^{\frac{\mu+k}{2}} t^{\frac{k-\mu}{2}}\int_{\lambda_0/\langle A_k(s)\rangle} ^{2\lambda_0/\langle A_k(s)\rangle}  \frac{1-\mathrm{e}^{-2\lambda (\phi_k(t)-\phi_k(s))}}{\phi_k(t)-\phi_k(s) }\,  \lambda^{q-1} \,\mathrm{d}\lambda  \\
&\geqslant \tfrac{B}{2} s^{\frac{\mu+k}{2}} t^{\frac{k-\mu}{2}} \frac{1-\mathrm{e}^{-2\lambda_0  \frac{\phi_k(t)-\phi_k(s)}{\langle A_k(s)\rangle}}}{\phi_k(t)-\phi_k(s) } \int_{\lambda_0/\langle A_k(s)\rangle} ^{2\lambda_0/\langle A_k(s)\rangle}    \lambda^{q-1} \,\mathrm{d}\lambda \\
& = \frac{B(2^q-1)\lambda_0^q}{2q} \,   s^{\frac{\mu+k}{2}} t^{\frac{k-\mu}{2}} \langle A_k(s)\rangle^{-q} \, \frac{1-\mathrm{e}^{-2\lambda_0  \frac{\phi_k(t)-\phi_k(s)}{\langle A_k(s)\rangle}}}{\phi_k(t)-\phi_k(s) } ,
\end{align*} with obvious modifications in the case $q=0$. The previous inequality implies \eqref{lower bound eta q}, provided that we show the validity of the inequality $$\frac{1-\mathrm{e}^{-2\lambda_0  \frac{\phi_k(t)-\phi_k(s)}{\langle A_k(s)\rangle}}}{\phi_k(t)-\phi_k(s) } \gtrsim \langle A_k(t)\rangle^{-1}.  $$ Hence, we need to prove this inequality. For $\phi_k(t)-\phi_k(s)\geqslant \frac{1}{2\lambda_0}\langle A_k(s)\rangle $, it holds  $$1-\mathrm{e}^{-2\lambda_0  \frac{\phi_k(t)-\phi_k(s)}{\langle A_k(s)\rangle}} \geqslant 1-\mathrm{e}^{-1}$$ and, consequently,
\begin{align*}
\frac{1-\mathrm{e}^{-2\lambda_0  \frac{\phi_k(t)-\phi_k(s)}{\langle A_k(s)\rangle}}}{\phi_k(t)-\phi_k(s) }  & \gtrsim \big(\phi_k(t)-\phi_k(s) \big)^{-1} \geqslant A_k(t)^{-1} \geqslant \langle A_k(t)\rangle^{-1}.
\end{align*} On the other hand, when $\phi_k(t)-\phi_k(s)\leqslant \frac{1}{2\lambda_0}\langle A_k(s)\rangle $, using the estimate $1-\mathrm{e}^{-\sigma}\geqslant \sigma/2$ for $\sigma\in [0,1]$, we get easily
\begin{align*}
\frac{1-\mathrm{e}^{-2\lambda_0  \frac{\phi_k(t)-\phi_k(s)}{\langle A_k(s)\rangle}}}{\phi_k(t)-\phi_k(s) }  & \geqslant  \frac{\lambda_0}{\langle A_k(s)\rangle} \geqslant  \frac{\lambda_0}{\langle A_k(t)\rangle}.
\end{align*} Therefore, the proof of \eqref{lower bound eta q} is completed.
\end{proof}

Next we prove an upper bound estimate in the special case $s=t$.
\begin{lemma} \label{Lemma lupper bound estimate xi q t=s} Let $n\geqslant 1$, $k\in[0,1)$, $\mu\geqslant 0$ and $\lambda_0>0$. If we assume  $q> (n-3)/2$, then, for $t\geqslant 1$ and $|x|\leqslant A_k(t) +R$ the following  upper bound estimate holds:
\begin{align}
\xi_q (t,t,x;k,\mu) &\leqslant B_2 \langle A_k(t)\rangle ^{-\frac{n-1}{2}}  \langle A_k(t) - |x|\rangle ^{\frac{n-3}{2}-q}. \label{upper bound xi q t=s}
\end{align}
Here $B_2$ is a positive constant depending only on $\lambda_0,q,R,k$ and $\langle y\rangle$ denotes the same function as in the statement of Lemma \ref{Lemma lower bound estimates xi q and eta q}.
\end{lemma}

\begin{proof}  Due to the representation
\begin{align*}
\xi_q(t,t,x;k,\mu) & = \int_0^{\lambda_0} \mathrm{e}^{-\lambda (A_k(t)+R)}  \varphi_\lambda(x) \lambda^q \,\mathrm{d}\lambda, 
\end{align*} the proof is exactly the same as in \cite[Lemma 2.7]{PTY20}.
\end{proof}

\subsection{Derivation of the iteration frame} \label{Subsection iteration frame}

In this section, we define the time -- dependent functional whose dynamic is studied in order to prove the blow -- up result. Then, we derive the iteration frame for this functional and a first lower bound estimate of logarithmic type.

For $t\geqslant 1$ we introduce the functional
\begin{align}\label{def functional mathcalU}
\mathcal{U}(t) \doteq t^{\frac{\mu-k}{2}} \int_{\mathbb{R}^n} u(t,x) \, \xi_q(t,t,x;k,\mu) \, \mathrm{d} x
\end{align}  for some $q>(n-3)/2$.  

From \eqref{fundamental inequality functional mathcalU}, \eqref{lower bound xi q} and \eqref{lower bound eta q}, it follows
\begin{align*}
\mathcal{U}(t) \gtrsim B_0 \varepsilon  \int_{\mathbb{R}^n} u_0(x) \, \mathrm{d}x + B_1 \varepsilon \, \frac{A_k(t)}{\langle A_k(t)\rangle } \int_{\mathbb{R}^n} u_1(x) \, \mathrm{d}x.
\end{align*} As we assume  both $u_0,u_1$  nonnegative and nontrivial, then, we find that 
\begin{align}\label{mathcalU > epsilon}
\mathcal{U}(t)\gtrsim \varepsilon
\end{align} for any $t\in [1,T)$, where the unexpressed multiplicative constant depends on $u_0,u_1$.
In the next proposition, we derive the iteration frame for the functional $\mathcal{U}$ for a given value of $q$.

\begin{proposition} \label{Proposition iteration frame} Let $n\geqslant 1$, $k\in [0,1)$ and $\mu\in [0,k]\cup [2-k,\infty)$. Let us consider $u_0\in H^1(\mathbb{R}^n)$ and $u_1\in L^2(\mathbb{R}^n)$ such that $\mathrm{supp}\, u_j \subset B_R$ for $j=0,1$ and for some $R>0$ and let $u$ be a local in time energy solution to \eqref{Semi EdeS k damped} on $[1,T)$ according to Definition \ref{Def energy sol}. If $\mathcal{U}$ is defined by \eqref{def functional mathcalU} with $q=(n-1)/2-1/p$, then, there exists a positive constant $C=C(n,p,R,k,\mu)$ such that
\begin{align} \label{iteration frame} 
\mathcal{U}(t)\geqslant C \langle A_k(t)\rangle^{-1}\int_1^t \frac{\phi_k(t)-\phi_k(s)}{s} \big(\log \langle A_k(s)\rangle\big)^{-(p-1)} (\mathcal{U}(s))^p\, \mathrm{d}s
\end{align} for any $t\in (1,T)$. 
\end{proposition}

\begin{proof}
By \eqref{def functional mathcalU}, applying H\"older's inequality we find
\begin{align*}
s^{\frac {k-\mu}2} \mathcal{U} (s) \leq \left(\int_{\mathbb{R}^n}|u(s,x)|^p \eta_q(t,s,x;k,\mu) \, \mathrm{d}x\right)^{1/p} \left(\int_{B_{R+A_k(s)}} \frac{\big(\xi_q(s,s,x;k,\mu)\big)^{p'}}{\big(\eta_q(t,s,x;k,\mu)\big)^{p'/p}} \, \mathrm{d}x\right)^{1/p'}.
\end{align*} Hence,
\begin{align} \label{intermediate lower bound int |u|^p eta_q}
\int_{\mathbb{R}^n}|u(s,x)|^p \eta_q(t,s,x;k,\mu) \, \mathrm{d}x \geqslant \big(s^{\frac {k-\mu}2} \mathcal{U} (s)\big)^p \left(\int_{B_{R+A_k(s)}} \frac{\big(\xi_q(s,s,x;k,\mu)\big)^{p/(p-1)}}{\big(\eta_q(t,s,x;k,\mu)\big)^{1/(p-1)}} \, \mathrm{d}x\right)^{-(p-1)}.
\end{align} Let us determine an upper bound  for the integral on the right -- hand side of \eqref{intermediate lower bound int |u|^p eta_q}. By using \eqref{upper bound xi q t=s} and \eqref{lower bound eta q}, we obtain
\begin{align*}
& \int_{B_{R+A_k(s)}}  \frac{\big(\xi_q(s,s,x;k,\mu)\big)^{p/(p-1)}}{\big(\eta_q(t,s,x;k,\mu)\big)^{1/(p-1)}} \, \mathrm{d}x \\ 
& \qquad \leqslant   B_1^{-\frac{1}{p-1}} B_2^{\frac{p}{p-1}}  s^{-\frac{\mu+k}{2(p-1)}} t^{-\frac{k-\mu}{2(p-1)}}  \langle A_k(s)\rangle ^{-\frac{n-1}{2}\frac{p}{p-1} +\frac{q}{p-1}} \langle A_k(t)\rangle ^{\frac{1}{p-1}} \int_{B_{R+A_k(s)}}  \langle A_k(s) - |x|\rangle ^{(\frac{n-3}{2}-q)\frac{p}{p-1}}  \mathrm{d}x \\
 & \qquad \leqslant   B_1^{-\frac{1}{p-1}} B_2^{\frac{p}{p-1}}   s^{-\frac{\mu+k}{2(p-1)}} t^{\frac{\mu-k}{2(p-1)}}  \langle A_k(t)\rangle ^{\frac{1}{p-1}}\langle A_k(s)\rangle ^{\frac{1}{p-1}(-\frac{n-1}{2}p+\frac{n-1}{2}-\frac{1}{p})}\int_{B_{R+A_k(s)}}  \langle A_k(s) - |x|\rangle ^{-1} \mathrm{d}x \\
 & \qquad \leqslant   B_1^{-\frac{1}{p-1}} B_2^{\frac{p}{p-1}}   s^{-\frac{\mu+k}{2(p-1)}} t^{\frac{\mu-k}{2(p-1)}}  \langle A_k(t)\rangle ^{\frac{1}{p-1}}\langle A_k(s)\rangle ^{\frac{1}{p-1}(-\frac{n-1}{2}p+\frac{n-1}{2}-\frac{1}{p})+n-1}  \log \langle A_k(s) \rangle  ,
\end{align*} where in the second inequality we used value of
 $q$ to get exactly $-1$ as power for the function in the integral.
Consequently, from \eqref{intermediate lower bound int |u|^p eta_q} we have
\begin{align*}
\int_{\mathbb{R}^n}|u(s,x)|^p \eta_q(t,s,x;k,\mu) & \, \mathrm{d}x \gtrsim \big(s^{\frac {k-\mu}2} \mathcal{U} (s)\big)^p s^{\frac{\mu+k}{2}} t^{\frac{k-\mu}{2}}  \langle A_k(t)\rangle ^{-1}\langle A_k(s)\rangle ^{-\frac{n-1}{2}(p-1)+\frac{1}{p}}  \big(\log \langle A_k(s) \rangle\big)^{-(p-1)} \\
& \gtrsim t^{\frac{k-\mu}{2}}   \langle A_k(t)\rangle ^{-1} s^{\frac {k}{2}(p+1)+\frac{\mu}{2}(1-p)} \langle A_k(s)\rangle ^{-\frac{n-1}{2}(p-1)+\frac{1}{p}}    \big(\log \langle A_k(s) \rangle\big)^{-(p-1)} \big(\mathcal{U} (s)\big)^p.
\end{align*}
Combining the previous lower bound estimate and \eqref{fundamental inequality functional mathcalU}, we arrive at
\begin{align*}
\mathcal{U}(t) & \geqslant t^{\frac{\mu-k}{2}} \int_1^t  (\phi_k(t)-\phi_k(s))   \int_{\mathbb{R}^n} |u(s,x)|^p \eta_q(t,s,x;k,\mu) \, \mathrm{d}x \, \mathrm{d}s \\
& \gtrsim    \langle A_k(t)\rangle ^{-1}  \int_1^t  (\phi_k(t)-\phi_k(s)) \, s^{\frac {k}{2}(p+1)+\frac{\mu}{2}(1-p)} \langle A_k(s)\rangle ^{-\frac{n-1}{2}(p-1)+\frac{1}{p}}    \big(\log \langle A_k(s) \rangle\big)^{-(p-1)} \big(\mathcal{U} (s)\big)^p\, \mathrm{d}s \\
& \gtrsim    \langle A_k(t)\rangle ^{-1}  \int_1^t  (\phi_k(t)-\phi_k(s))  \langle A_k(s)\rangle ^{\frac {k(p+1)}{2(1-k)}-\frac{\mu(p-1)}{2(1-k)}-\frac{n-1}{2}(p-1)+\frac{1}{p}}    \big(\log \langle A_k(s) \rangle\big)^{-(p-1)} \big(\mathcal{U} (s)\big)^p\, \mathrm{d}s\\
& \gtrsim    \langle A_k(t)\rangle ^{-1}  \int_1^t  (\phi_k(t)-\phi_k(s))  \langle A_k(s)\rangle ^{-\left(\frac{n-1}{2}+\frac{\mu-k}{2(1-k)}\right)p+\left(\frac{n-1}{2}+\frac{\mu+k}{2(1-k)}\right)+\frac 1p}    \big(\log \langle A_k(s) \rangle\big)^{-(p-1)} \big(\mathcal{U} (s)\big)^p\, \mathrm{d}s,
\end{align*} where in third step we used $s= (1-k)^{\frac{1}{1-k}}(A_k(s)+\phi_k(1))^{\frac{1}{1-k}}\approx \langle A_k(s)\rangle^{\frac{1}{1-k}}$ for $s\geqslant 1$. Since $p=p_0\big(k,n+\frac{\mu}{1-k}\big)$ from \eqref{intro equation critical exponent shifted} it follows
\begin{align}\label{equation critical exponent intermediate}
-\left(\tfrac{n-1}{2}+\tfrac{\mu-k}{2(1-k)}\right)p+\left(\tfrac{n-1}{2}+\tfrac{\mu+k}{2(1-k)}\right)+\tfrac 1p= -1-\tfrac{k}{1-k}=-\tfrac{1}{1-k},
\end{align} then, plugging \eqref{equation critical exponent intermediate} in the above lower bound estimate for $\mathcal{U}(t)$ it yields
\begin{align*}
\mathcal{U}(t) &  \gtrsim    \langle A_k(t)\rangle ^{-1}  \int_1^t  (\phi_k(t)-\phi_k(s))  \langle A_k(s)\rangle ^{-\frac{1}{1-k}}    \big(\log \langle A_k(s) \rangle\big)^{-(p-1)} \big(\mathcal{U} (s)\big)^p\, \mathrm{d}s \\
&  \gtrsim    \langle A_k(t)\rangle ^{-1}  \int_1^t  \frac{\phi_k(t)-\phi_k(s)}{s}  \big(\log \langle A_k(s) \rangle\big)^{-(p-1)} \big(\mathcal{U} (s)\big)^p\, \mathrm{d}s,
\end{align*} which is exactly \eqref{iteration frame}. Therefore,  the proof is completed.
\end{proof}

\begin{lemma} \label{Lemma lower bound int |u|^p} Let $n\geqslant 1$, $k\in [0,1)$ and $\mu\in [0,k]\cup [2-k,\infty)$. Let us consider $u_0\in H^1(\mathbb{R}^n)$ and $u_1\in L^2(\mathbb{R}^n)$ such that $\mathrm{supp}\, u_j \subset B_R$ for $j=0,1$ and for some $R>0$ and let $u$ be a local in time energy solution to \eqref{Semi EdeS k damped} on $[1,T)$ according to Definition \ref{Def energy sol}. Then, there exists a positive constant $K=K(u_0,u_1,n,p,R,k,\mu)$ such that the lower bound estimate
\begin{align} \label{lower bound int |u|^p}
\int_{\mathbb{R}^n} |u(t,x)|^p \, \mathrm{d}x \geqslant K \varepsilon^p \langle A_k(t)\rangle^{(n-1)(1-\frac{p}{2})+\frac{(k-\mu)p}{2(1-k)}}
\end{align} holds for any $t\in (1,T)$. 
\end{lemma}

\begin{proof} We modify of the proof of Lemma 5.1 in \cite{WakYor18} accordingly to our model. Let us fix $q>(n-3)/2 +1/p'$. Combining \eqref{def functional mathcalU}, \eqref{mathcalU > epsilon} and H\"older's inequality, it results
\begin{align*}
\varepsilon t^{\frac{k-\mu}{2}} &\lesssim t^{\frac{k-\mu}{2}} \mathcal{U}(t) =  \int_{\mathbb{R}^n} u(t,x) \, \xi_q(t,t,x;k,\mu) \, \mathrm{d} x \\ & \leqslant \left(\int_{\mathbb{R}^n}|u(t,x)|^p\, \mathrm{d}x\right)^{1/p} \left(\int_{B_{R+A_k(t)}}\big(\xi_q(t,t,x;k,\mu\big)^{p'} \mathrm{d}x\right)^{1/p'}.
\end{align*} Hence,
\begin{align} \label{lower bound int |u|^p intermediate}
\int_{\mathbb{R}^n}|u(t,x)|^p\, \mathrm{d}x \gtrsim \varepsilon^p t^{\frac{k-\mu}{2}p} \left(\int_{B_{R+A_k(t)}}\big(\xi_q(t,t,x;k,\mu\big)^{p'} \mathrm{d}x\right)^{-(p-1)}.
\end{align} Let us determine an upper bound estimates for the integral of $\xi_q(t,t,x;k,\mu)^{p'}$. By using \eqref{upper bound xi q t=s}, we have
\begin{align*}
\int_{B_{R+A_k(t)}}\big(\xi_q(t,t,x;k,\mu\big)^{p'} \mathrm{d}x & \lesssim \langle A_k(t)\rangle ^{-\frac{n-1}{2}p'} \int_{B_{R+A_k(t)}} \langle A_k(t) - |x|\rangle ^{(n-3)p'/2-p'q} \, \mathrm{d}x \\
& \lesssim \langle A_k(t)\rangle ^{-\frac{n-1}{2}p'} \int_0^{R+A_k(t)} r^{n-1} \langle A_k(t) - r\rangle ^{(n-3)p'/2-p'q} \, \mathrm{d}r \\
& \lesssim \langle A_k(t)\rangle ^{-\frac{n-1}{2}p'+n-1} \int_0^{R+A_k(t)}  \langle A_k(t) - r\rangle ^{(n-3)p'/2-p'q} \, \mathrm{d}r.
\end{align*} Performing the change of variable $A_k(t)-r=\varrho$, one gets
\begin{align*}
\int_{B_{R+A_k(t)}}\big(\xi_q(t,t,x;k,\mu\big)^{p'} \mathrm{d}x & \lesssim  \langle A_k(t)\rangle ^{-\frac{n-1}{2}p'+n-1} \int^{A_k(t)}_{-R}  (3+|\varrho|)^{(n-3)p'/2-p'q} \, \mathrm{d}\varrho \\
& \lesssim  \langle A_k(t)\rangle ^{-\frac{n-1}{2}p'+n-1}
\end{align*} because of $(n-3)p'/2 -p'q<-1$.
If we combine this upper bound estimates for the integral of $\xi_q(t,t,x;k,\mu)^{p'}$, the inequality \eqref{lower bound int |u|^p intermediate} and we employ $t\approx\langle A_k(t)\rangle^{\frac{1}{1-k}}$ for $t\geqslant 1$, then, we arrive at \eqref{lower bound int |u|^p}. This completes the proof.
\end{proof}

In Proposition \ref{Proposition iteration frame}, we derive the iteration frame for $\mathcal{U}$. In the next result, we shall prove a first lower bound estimate of logarithmic type for $\mathcal{U}$, as base case for the iteration argument.

\begin{proposition} \label{Proposition first logarithmic lower bound mathcalU}  Let $n\geqslant 1$, $k\in [0,1)$ and $\mu\in [0,k]\cup [2-k,\infty)$. Let us consider $u_0\in H^1(\mathbb{R}^n)$ and $u_1\in L^2(\mathbb{R}^n)$ such that $\mathrm{supp}\, u_j \subset B_R$ for $j=0,1$ and for some $R>0$ and let $u$ be a local in time energy solution to \eqref{Semi EdeS k damped} on $[1,T)$ according to Definition \ref{Def energy sol}. Let $\mathcal{U}$ be defined by \eqref{def functional mathcalU} with $q=(n-1)/2-1/p$. Then, for $t\geqslant 3/2$ the functional $\mathcal{U}(t)$ fulfills
\begin{align} \label{first logarithmic lower bound mathcalU}
\mathcal{U}(t) \geqslant M \varepsilon^p \log\left(\tfrac{2t}{3}\right),
\end{align} where  the positive constant $M$ depends on $u_0,u_1,n,p,R,k,\mu$.
\end{proposition}

\begin{proof}
From \eqref{fundamental inequality functional mathcalU} it results
\begin{align*}
\mathcal{U}(t) & \geqslant  t^{\frac{\mu-k}{2}} \int_1^t  (\phi_k(t)-\phi_k(s))   \int_{\mathbb{R}^n} |u(s,x)|^p \eta_q(t,s,x;k,\mu) \, \mathrm{d}x \, \mathrm{d}s.
\end{align*}  Consequently, applying \eqref{lower bound eta q} first and then \eqref{lower bound int |u|^p}, we find
\begin{align*}
\mathcal{U}(t) & \geqslant B_1  \langle A_k(t)\rangle ^{-1} \int_1^t  (\phi_k(t)-\phi_k(s)) \, s^{\frac{\mu+k}{2}} \langle A_k(s)\rangle ^{-q} \int_{\mathbb{R}^n} |u(s,x)|^p  \, \mathrm{d}x \, \mathrm{d}s \\
& \geqslant B_1 K \varepsilon^p \langle A_k(t)\rangle ^{-1} \int_1^t  (\phi_k(t)-\phi_k(s)) \, s^{\frac{\mu+k}{2}}   \langle A_k(s)\rangle^{-q+(n-1)(1-\frac{p}{2})+\frac{(k-\mu)p}{2(1-k)}} \,  \mathrm{d}s \\
& \gtrsim \varepsilon^p \langle A_k(t)\rangle ^{-1} \int_1^t  (\phi_k(t)-\phi_k(s))    \langle A_k(s)\rangle^{\frac{\mu+k}{2(1-k)}-\frac{n-1}{2}+\frac{1}{p}+(n-1)(1-\frac{p}{2})+\frac{(k-\mu)p}{2(1-k)}}  \, \mathrm{d}s \\
& \gtrsim  \varepsilon^p \langle A_k(t)\rangle ^{-1} \int_1^t  (\phi_k(t)-\phi_k(s))  \langle A_k(s)\rangle ^{-\left(\frac{n-1}{2}+\frac{\mu-k}{2(1-k)}\right)p+\left(\frac{n-1}{2}+\frac{\mu+k}{2(1-k)}\right)+\frac 1p}    \, \mathrm{d}s \\
& \gtrsim  \varepsilon^p \langle A_k(t)\rangle ^{-1} \int_1^t  (\phi_k(t)-\phi_k(s))  \langle A_k(s)\rangle ^{-\frac{1}{1-k}}    \, \mathrm{d}s  \gtrsim  \varepsilon^p \langle A_k(t)\rangle ^{-1} \int_1^t  \frac{\phi_k(t)-\phi_k(s)}{s} \, \mathrm{d}s. 
\end{align*}  Integrating by parts, we obtain
\begin{align*}
 \int_1^t  \frac{\phi_k(t)-\phi_k(s)}{s} \, \mathrm{d}s &  =  \big(\phi_k(t)-\phi_k(s)\big)\log s \, \Big|^{s=t}_{s=1} +\int_1^t  \phi_k'(s) \log s \, \mathrm{d}s \\ &= \int_1^t s^{-k} \log s \, \mathrm{d}s \geqslant t^{-k} \int_1^t  \log s \, \mathrm{d}s.
\end{align*} Consequently, for $t\geqslant 3/2$ 
\begin{align*}
\mathcal{U}(t) &   \gtrsim  \varepsilon^p \langle A_k(t)\rangle ^{-1} t^{-k} \int_1^t  \log s \, \mathrm{d}s \geqslant \varepsilon^p \langle A_k(t)\rangle ^{-1} t^{-k} \int_{2t/3}^t  \log s \, \mathrm{d}s \geqslant (1/3) \varepsilon^p \langle A_k(t)\rangle ^{-1} t^{1-k}   \log (2t/3) \\
& \gtrsim  \varepsilon^p  \log (2t/3) ,
\end{align*} where in the last line we applied $t\approx \langle A_k(t)\rangle ^{\frac{1}{1-k}}$ for $t\geqslant 1$. Thus, the proof is over.
\end{proof}

In order to conclude the proof of Theorem \ref{Theorem critical case p0} it remains to use an iteration argument together with a slicing procedure for the domain of integration. This procedure consists in determining a sequence of lower bound estimates for $\mathcal{U}(t)$ (indexes by $j\in\mathbb{N}$) and, then, proving that $\mathcal{U}(t)$ may not be finite for $t$ over a certain $\varepsilon$ -- dependent threshold by taking the limit as $j\to \infty$. Since the iteration frame \eqref{iteration frame} and the first lower bound estimate \eqref{first logarithmic lower bound mathcalU} are formally identical to those in \cite[Section 2.3]{PTY20} (of course, for different values of the critical exponent $p$), the iteration argument can be rewritten verbatim as in \cite[Section 2.4]{PTY20}.

 Finally, we show how the previous steps can be adapted to the treatment of the case $\mu\in[0,k]$. According to Remark \ref{Remark transf v} , through the transformation $v(t,x)=t^{\mu -1} u(t,x)$, we may consider the transformed semilinear Cauchy problem \eqref{Semi EdeS k damped transf v} for $v$. Note that $v_0\doteq u_0$ and $v_1\doteq u_1+(1-\mu)u_0$ satisfies the same assumptions for $u_0$ and $u_1$ in the statement of Theorem \ref{Theorem critical case p0} in this case (nonnegativeness and nontriviality, compactly supported and belongingness to the energy space $H^1(\mathbb{R}^n)\times L^2(\mathbb{R}^n)$). Of course, we may introduce the auxiliary function $\xi_q(t,s,x;k,2-\mu),\xi_q(t,s,x;k,2-\mu)$ as in \eqref{def xi q}, \eqref{def eta q} replacing $\mu$ by $2-\mu$. In Corollary \ref{Corollary fund ineq}, nevertheless,  we have to replace the fundamental identity \eqref{fundametal integral equality} by 
\begin{align*}
\int_{\mathbb{R}^n} v(t,x) \, \xi_q(t,t,x;k,2-\mu) \, \mathrm{d}x & = \varepsilon \int_{\mathbb{R}^n} v_0(x) \,  \xi_q(t,1,x;k,2-\mu)  \, \mathrm{d}x 
+ \varepsilon \,  A_k(t) \int_{\mathbb{R}^n} v_1(x) \, \eta_q(t,s,x;k,2-\mu)  \, \mathrm{d}x \notag \\ & \quad +\int_1^t  (\phi_k(t)-\phi_k(s)) s^{(1-\mu)(p-1)}  \int_{\mathbb{R}^n} |v(s,x)|^p \eta_q(t,s,x;k,2-\mu) \, \mathrm{d}x \, \mathrm{d}s.
\end{align*} As we have already pointed out in Remark \ref{Remark transf v}, the estimates in \eqref{lower bound estimate y0(t,s;lambda,k)} and \eqref{lower bound estimate y_1(t,s;lambda,k)} holds true in this case with $2-\mu$ instead of $\mu$ (we recall that this was the actual reason to consider the transformed problem in place of the original one). Moreover, also the lower bound estimate in \eqref{lower bound int |u|^p} is valid for $v$, provided that we replace $\mu$ by $2-\mu$. Accordingly to what we have just remarked, the suitable time -- dependent functional to study for the transformed problem is $$\mathcal{V}(t)\doteq t^{1-\frac{\mu+k}{2}} \int_{\mathbb{R}^n} v(t,x) \, \xi_q(t,t,x;k,2-\mu) \, \mathrm{d} x.$$ In fact, $\mathcal{V}$ satisfies $\mathcal{V}(t)\gtrsim \varepsilon$ for $t\in [1,T)$ and, furthermore, it is possible to derive for $\mathcal{V}$  completely analogous iteration frame and  first logarithmic lower bound, respectively,  as the ones for $\mathcal{U}$ in \eqref{iteration frame} and \eqref{first logarithmic lower bound mathcalU}, respectively. We point out that both for the iteration frame and for the first  logarithmic lower bound estimate the time -- dependent factor $t^{(1-\mu)(p-1)}$ in the nonlinearity compensates the modifications due to the replacement of $\mu$ by $2-\mu$ in the proofs of Propositions \ref{Proposition iteration frame} and \ref{Proposition first logarithmic lower bound mathcalU}.

\section{Critical case: part II}  \label{Section critical case p1}

In Section \ref{Section subcritical case}, we derived the upper bound for the lifespan in the subcritical case, whereas in Section \ref{Section critical case p0} we studied the critical case $p=p_0\big(k,n+\frac{\mu}{1-k}\big)$. It remains to consider the critical case $p=p_1(k,n)$, that is, when $\mu\geqslant \mu_0(k,n)$. In this section, we are going to prove  Theorem \ref{Theorem critical case p1}. In this critical case, our approach will be based on a basic iteration argument combined with the slicing procedure introduced for the first time in the paper \cite{AKT00}. The parameters characterizing the slicing procedure are given by the sequence $\{\ell_j\}_{j\in\mathbb{N}}$, where $\ell_j\doteq 2-2^{-(j+1)}$.

As time -- depending functional we consider the same one studied in Section \ref{Section subcritical case}, namely, $U_0$ defined in \eqref{def U}. Hence, since $p=p_1(k,n)$ is equivalent to the condition 
\begin{align}\label{condition p=p1 equiv}
(1-k)n(p-1)=2,
\end{align}
 we can rewrite \eqref{iter fram subcrit U0} as 
\begin{align}\label{iteration frame 2nd crit case}
U_0(t)& \geqslant  C \int_1^t \tau^{-\mu}\int_1^\tau s^{\mu -2}  (U_0(s))^p\,  \mathrm{d}s \,  \mathrm{d}\tau  
\end{align} for any $t\in (1,T)$ and  for a suitable positive constant $C>0$. Let us underline that \eqref{iteration frame 2nd crit case} will be the iteration frame in the iteration procedure for the critical case $p=p_1(k,n)$. 

We know that $U_0(t)\geqslant K \varepsilon$ for any $t\in (1,T)$ and for  a suitable positive constant $K$, provided that $u_0,u_1$ are nonnegative, nontrivial and compactly supported (cf. the estimate in \eqref{lb estimate U0 trivial}). Thus,
\begin{align}
U_0(t) & \geqslant C K^p \varepsilon^p \int_1^t \tau^{-\mu}\int_1^\tau s^{\mu -2}  \,  \mathrm{d}s \,  \mathrm{d}\tau   \geqslant  C K^p \varepsilon^p \int_1^t \tau^{-\mu-2}\int_1^\tau (s-1)^{\mu }  \,  \mathrm{d}s \,  \mathrm{d}\tau \notag \\
& = \frac{C K^p\varepsilon^p}{ \mu+1}  \int_1^t \tau^{-\mu-2} (\tau-1)^{\mu +1} \,  \mathrm{d}\tau \geqslant \frac{C K^p\varepsilon^p}{ \mu+1} \int_{\ell_0}^t \tau^{-\mu-2} (\tau-1)^{\mu +1} \,  \mathrm{d}\tau \notag \\ &   \geqslant \frac{C K^p\varepsilon^p}{ 3^{\mu+1}(\mu+1)}  \int_{\ell_0}^t \tau^{-1} \,  \mathrm{d}\tau \geqslant   \frac{C K^p\varepsilon^p}{ 3^{\mu+1}(\mu+1)}  \log\left(\frac{t}{\ell_0}\right) \label{1st lower bound U p=p1}
\end{align} for $t \geqslant \ell_0=3/2$, where we used $\tau\leqslant 3(\tau-1)$ for $\tau \geqslant \ell_0$ in the second last step.

Therefore, by using recursively \eqref{iteration frame 2nd crit case}, we prove now the sequence of lower bound estimates
\begin{align}\label{lower bound U j p=p1}
U_0(t)\geqslant K_j  \left(\log \left(\frac{t}{\ell_j}\right)\right)^{\sigma_j} \qquad \mbox{for} \ t\geqslant \ell_j
\end{align} for any $j\in \mathbb{N}$, where $\{K_j\}_{j\in\mathbb{N}}$, $\{\sigma\}_{j\in\mathbb{N}}$ are sequences of positive reals that we determine afterwards in the inductive step.

Clearly \eqref{lower bound U j p=p1} for $j=0$ holds true thanks to \eqref{1st lower bound U p=p1}, provided that $K_0= (CK^p  \varepsilon^p)/(3^{\mu+1}(\mu+1))$ and $\sigma_0=1$. Next we show the validity of \eqref{lower bound U j p=p1} by using an inductive argument. Assuming that  \eqref{lower bound U j p=p1} is satisfied for some $j\geqslant 0$, we prove \eqref{lower bound U j p=p1} for $j+1$. According to this purpose, we plug \eqref{lower bound U j p=p1} in \eqref{iteration frame 2nd crit case}, so, after shrinking the domain of integration, we get
\begin{align*}
U_0(t) & \geqslant CK_j^p  \int_{\ell_j}^t \tau^{-\mu}\int_{\ell_j}^\tau s^{\mu -2}    \left(\log \left( \tfrac{s}{\ell_j}\right)\right)^{\sigma_j p}  \mathrm{d}s \,  \mathrm{d}\tau  
\end{align*} for $t\geqslant \ell_{j+1}$. If we shrink the domain of integration to $[(\ell_j/\ell_{j+1})\tau,\tau]$ in the $s$ -- integral (this operation is possible for $\tau\geqslant \ell_{j+1}$), we find
\begin{align*}
U_0(t) & \geqslant CK_j^p  \int_{\ell_{j+1}}^t \tau^{-\mu-2}\int_{\tfrac{\ell_j \tau}{\ell_{j+1}}}^\tau s^{\mu }    \left(\log \left( \tfrac{s}{\ell_j}\right)\right)^{\sigma_j p}  \mathrm{d}s \,  \mathrm{d}\tau   \\ &\geqslant CK_j^p  \int_{\ell_{j+1}}^t \tau^{-\mu-2} \left(\log \left( \tfrac{\tau }{\ell_{j+1}}\right)\right)^{\sigma_j p} \int_{\tfrac{\ell_j \tau}{\ell_{j+1}}}^\tau \left(s-\tfrac{\ell_j}{\ell_{j+1}} \tau\right)^{\mu }     \mathrm{d}s \,  \mathrm{d}\tau  \\
& = CK_j^p (\mu+1)^{-1} \left(1-\tfrac{\ell_j}{\ell_{j+1}} \right)^{\mu +1}      \int_{\ell_{j+1}}^t \tau^{-1} \left(\log \left( \tfrac{\tau }{\ell_{j+1}}\right)\right)^{\sigma_j p}  \mathrm{d}\tau \\
& \geqslant 2^{-(j+3)(\mu+1)} CK_j^p (\mu+1)^{-1}  (1+p\sigma_j)^{-1}     \left(\log \left( \tfrac{t }{\ell_{j+1}}\right)\right)^{\sigma_j p+1}
\end{align*} for $t\geqslant \ell_{j+1}$, where in the last step we applied the inequality $1-\ell_j/\ell_{j+1}>2^{-(j+3)}$. Hence, we proved \eqref{lower bound U j p=p1} for $j+1$ provided that 
\begin{align*}
K_{j+1}\doteq 2^{-(j+3)(\mu+1)} C (\mu+1)^{-1}  (1+p\sigma_j)^{-1}  K_j^p \quad \mbox{and} \quad \sigma_{j+1} \doteq 1 + \sigma_j p .
\end{align*}

Let us establish a suitable lower bound for $K_j$.  Using iteratively the relation $\sigma_j=1+p\sigma_{j-1}$ and the initial exponent $\sigma_0=1$, we have
\begin{align}\label{explit expressions sigmaj}
\sigma_j & = \sigma_0 p^j +\sum_{k=0}^{j-1} p^k =  \tfrac{p^{j+1}-1}{p-1}.
\end{align} In particular, the inequality $\sigma_{j-1}p+1= \sigma_j\leqslant p^{j+1}/(p-1) $ yields
\begin{align} \label{lower bound Kj no.1}
K_j \geqslant L\, \big(2^{\mu+1} p\big)^{-j} K^p_{j-1}
\end{align} for any $j\geqslant 1$, where $L\doteq {2^{-2(\mu+1)}} C (\mu+1)^{-1} (p-1)/p$. Applying the logarithmic function to both sides of  \eqref{lower bound Kj no.1} and using the resulting inequality iteratively, we obtain
\begin{align*}
\log K_j & \geqslant p \log K_{j-1} -j \log \big(2^{\mu+1} p\big)+\log L \\
& \geqslant \ldots \geqslant p^j \log K_0 -\Bigg(\sum_{k=0}^{j-1}(j-k)p^k \Bigg)\log \big(2^{\mu+1} p\big)+\Bigg(\sum_{k=0}^{j-1} p^k \Bigg)\log L  \\
& = \! p^j \left(\!\log \left(\frac{CK^p \varepsilon^p}{3^{\mu+1} (\mu+1)} \right) -\frac{p\log \big(2^{\mu+1} p\big)}{(p-1)^2}+\frac{\log L }{p-1}\!\right)\!+\!\left( \frac{j}{p-1}+\frac{p}{(p-1)^2}\!\right)\log \big(2^{\mu+1} p\big)-\frac{\log L}{p-1},
\end{align*} where we applied again the identities in \eqref{summation identities}. Let us define $j_2=j_2(n,p,k,\mu)$ as the smallest nonnegative integer such that $$j_2\geqslant \frac{\log L}{\log \big(2^{\mu+1} p\big)}-\frac{p}{p-1}.$$ Consequently, for any $j\geqslant j_2$ the following estimate holds
\begin{align} \label{lower bound Kj no.2}
\log K_j & \geqslant  p^j \left(\log \left( \frac{CK^p \varepsilon^p}{3^{\mu+1} (\mu+1)}\right) -\frac{p\log \big(2^{\mu+1} p\big)}{(p-1)^2}+\frac{\log L}{p-1}\right) = p^j \log ( N \varepsilon^p),
\end{align} where $N\doteq 3^{-(\mu+1)}CK^p (\mu+1)^{-1} \big(2^{\mu+1}p\big)^{-p/(p-1)^2}L^{1/(p-1)}$. 

Combining \eqref{lower bound U j p=p1}, \eqref{explit expressions sigmaj} and \eqref{lower bound Kj no.2}, we arrive at
\begin{align*}
U_0(t)&\geqslant \exp \left( p^j\log(N\varepsilon^p)\right)  \left(\log \left(\tfrac t{\ell_j}\right)\right)^{\sigma_j}  \\ & \geqslant \exp \left( p^j\log(N\varepsilon^p)\right)  \left( \tfrac 12 \log t \right)^{(p^{j+1}-1)/(p-1)} \\
&= \exp \left( p^j\log\left( 2^{-p/(p-1)}N\varepsilon^p \left(\log t \right)^{p/(p-1)}\right) \right)  \left( \tfrac 12 \log  t \right)^{-1/(p-1)} 
\end{align*} for $t\geqslant 4$ and for any $j\geqslant j_2$, where we employed the inequality $\log (t/ \ell_j)\geqslant \log(t/2) \geqslant (1/2) \log t$ for $t\geqslant 4$.
Introducing the notation $H(t,\varepsilon)\doteq 2^{-p/(p-1)}N\varepsilon^p \left(\log t\right)^{p/(p-1)}$, the previous estimate may be rewritten as
 \begin{align}\label{final lower bound U}
 U_0(t)&\geqslant \exp \big( p^j \log H(t,\varepsilon)\big)  \left( \tfrac 12 \log t \right)^{-1/(p-1)} 
 \end{align} for $t\geqslant 4$ and any $j\geqslant j_2$.

 If we fix  $\varepsilon_0=\varepsilon_0(n,p,k,\mu,R,u_0,u_1)$ such that
 \begin{align*}
 \exp \left(2N^{-(1-p)/p}\varepsilon_0^{-(p-1)}\right)\geqslant 4,
 \end{align*}
 then, for any $\varepsilon\in (0,\varepsilon_0]$ and for $t> \exp \left(2N^{-(1-p)/p}\varepsilon^{-(p-1)}\right)$ we have $t\geqslant 4$ and $H(t,\varepsilon)>1$. Therefore, for any $\varepsilon\in (0,\varepsilon_0]$ and for $t> \exp \left(2N^{-(1-p)/p}\varepsilon^{-(p-1)}\right)$ letting $j\to \infty$ in \eqref{final lower bound U} we see that the lower bound for $U_0(t)$ blows up and, consequently, $U_0(t)$ may not be finite as well. Summarizing, we proved  that $U_0$ blows up in finite time and, moreover, we showed the upper bound estimate for the lifespan $$T(\varepsilon)\leqslant \exp \left(2N^{-(1-p)/p}\varepsilon^{-(p-1)}\right).$$ 
Hence the proof of Theorem \ref{Theorem critical case p1} in the critical case $p=p_1(k,n)$ is complete.


\section{Final remarks}

According to the results we obtained in Theorems \ref{Theorem subcritical case}, \ref{Theorem critical case p0} and \ref{Theorem critical case p1} it is quite natural to conjecture that $$\max\big\{p_0\big(k,n+\tfrac{\mu}{1-k}\big),p_1(k,n)\big\}$$ is the critical exponent for the semilinear Cauchy problem \eqref{Semi EdeS k damped}, although the global existence of small data solutions is completely open in the supercritical case. Furthermore, this exponent is consistent with other models studied in the literature.

In the flat case $k=0$, this exponent coincide with $\max\{p_{\mathrm{Str}}(n+\mu),p_{\mathrm{Fuj}}(n)\}$ which in many subcases has been showed to be optimal in the case of semilinear wave equation with time -- dependent scale -- invariant damping, see \cite{D15,DLR15,DL15,NPR16,IS17,TuLin17,PalRei18,PT18,Pal18odd,Pal18even,D20} and references therein for further details.

 On the other hand, in the undamped case $\mu=0$ (that is, for the semilinear wave equation with speed of propagation $t^{-k}$) the exponent $\max\{p_0(k,n),p_1(k,n)\}$ is consistent with the result for the generalized semilinear Tricomi equation (i.e., the semilinear wave equation with speed of propagation $t^\ell$, where $\ell>0$) obtained in the recent works \cite{HeWittYin17,HeWittYin17Crit,HeWittYin18,LinTu19}.
 
  Clearly, in the very special case $\mu=0$ and $k=0$, our result is nothing but a blow-up result for the classical semilinear wave equation for exponents below $p_{\mathrm{Str}}(n)$, which is well -- known to be optimal (for a detailed historical overview on Strauss' conjecture and a complete list of references we address the reader to the introduction of the paper \cite{TW11}).

As we have already explained in the introduction, for $\mu=2$ and $k=2/3$ the equation in \eqref{Semi EdeS k damped} is the semilinear wave equation in the Einstein -- de Sitter spacetime. In particular, our result is a natural generalization of the results in \cite{GalYag17EdS,PTY20}.

Furthermore, we underline explicitly the fact that the exponent $p_0\big(k,n+\tfrac{\mu}{1-k}\big)$ for \eqref{Semi EdeS k damped} is obtained by the corresponding exponent in the not damped case $\mu=0$ via a formal shift in the dimension of magnitude $\tfrac{\mu}{1-k}$. This phenomenon is due to the threshold nature of the time -- dependent coefficient of the damping term and it has been widely observed in the special case $k=0$ not only for the semilinear Cauchy problem with power nonlinearity but also with nonlinarity of derivative type $|u_t|^p$ (see \cite{PT19}) or weakly coupled system (see \cite{CP19,Pal19,PT19}).

Finally, we have to point out that after the completion of the final version of this work, we found out the existence of the paper \cite{TW20}, where the same model is considered. We stress that the approach we used in the critical case is completely different, and that we slightly improved their result, by removing the assumption on the size of the support of the Cauchy data (cf. \cite[Theorem 2.3]{TW20}), even though we might not cover the full range $\mu \in [0,\mu_0(k,n)]$ in the critical case due to the assumption $\mu\not\in (k,2-k)$.

\section*{Acknowledgments}

A. Palmieri 
is supported by the GNAMPA project `Problemi stazionari e di evoluzione nelle equazioni di campo nonlineari dispersive'. The author would like to  acknowledge Karen Yagdjian (UTRGV), who first introduced him to the model considered in this work.

\appendix

\end{document}